\documentclass{amsart}
\usepackage{graphicx}
\usepackage{a4wide}
\usepackage{amssymb}

\newtheorem{theorem}{Theorem}[section]    
\newtheorem{lemma}[theorem]{Lemma}          
\newtheorem{proposition}[theorem]{Proposition}

\newtheorem{corollary}[theorem]{Corollary} 

\theoremstyle{definition}

\newtheorem{remark}[theorem]{Remark}

\newtheorem{example}[theorem]{Example}    
\numberwithin{equation}{section}
\newcommand{\F}{\mathcal F_{ob} }
\newcommand{\Int}{\textrm{Int}}
\newcommand{\mF}{\mathcal{F}}
\newcommand{\R}{\mathbb{R}}

\newcommand{\Z}{\mathbb{Z}}
\newcommand{\sgn}{{\tt sgn}}
\title[Self-linking number and braid index]{On a relation between the self-linking number and the braid index of closed braids in open books}
\author{Tetsuya Ito}
\address{Research Institute for Mathematical Sciences, Kyoto university, Kyoto, 606-8502, Japan}
\email{tetitoh@kurims.kyoto-u.ac.jp}
\urladdr{http://www.kurims.kyoto-u.ac.jp/~tetitoh/}

\date{\today} 

\begin{document}

\begin{abstract}
We prove a generalization of the Jones-Kawamuro conjecture that relates the self-linking number and the braid index of closed braids, for planar open books with certain additional conditions and modifications. We show that our result is optimal in some sense by giving several counter examples for naive generalizations of the Jones-Kawamuro conjecture.
\end{abstract}

\maketitle

\section{Introduction}

In a seminal paper \cite{j}, V. Jones observed formulae that relate the HOMFLY polynomial to the Alexander polynomial and the algebraic linking number (exponent sum) for closed 3- and 4-braids \cite[(8.4) and (8.10)] {j}. This leads him to write {\it ``Formulae (8.4) and (8.10) lend some weight to the possibility that the exponent sum in a minimal braid representation is a knot invariant''.}

This question, whether the algebraic linking number yields a topological knot invariant when a knot is represented as a closed braid of the minimal braid index, is later called \emph{Jones' conjecture}. In \cite{k1} K. Kawamuro proposed a generalization of Jones' conjecture which we call the \emph{Jones-Kawamuro conjecture}:
if two closed braids $\widehat{\alpha}$ and $\widehat{\beta}$ represent the same oriented link $L$, the inequality
\begin{equation}
\label{eqn:JKconj}
 |w(\widehat{\alpha})-w(\widehat{\beta})| \leq n(\widehat{\alpha}) + n(\widehat{\beta}) -2 b(L)
\end{equation}
holds. Here $w$ and $n$ denotes the algebraic linking number and the braid index of a closed braid, and $b(L)$ is the minimal braid index of $L$, the minimum number of strands needed to represent $L$ as a closed braid. Recently, the Jones-Kawamuro conjecture (\ref{eqn:JKconj}) was solved affirmatively by Dynnikov-Prasolov \cite{dp} and LaFountain-Menasco \cite{lm}, by different but related methods.

By Bennequin's formula $sl(\widehat{\alpha})= w(\widehat{\alpha})-n(\widehat{\alpha})$ of the self-linking number of a closed braid \cite{be}, the inequality (\ref{eqn:JKconj}) implies
\begin{equation}
\label{eqn:JKconj2} 
|sl(\widehat{\alpha})-sl(\widehat{\beta})| \leq 2 (\max \{n(\widehat{\alpha}),n(\widehat{\beta})\} -b(L)). 
\end{equation}
Thus, in a point of view of contact geometry, the Jones-Kawamuro conjecture can be understood as an interaction between the self-linking number and the braid indices. In particular, Jones' conjecture states a surprising phenomenon that the self-linking number, the most fundamental \emph{transverse} knot invariant, yields a \emph{topological} knot invariant when it attains the minimal braid index.

In this paper we prove a generalization of the Jones-Kawamuro conjecture for planar open books, under some additional assumptions and conditions. Our main theorem includes the original Jones-Kawamuro conjecture as its special case, and provides an optimal generalization of the Jones-Kawamuro conjecture for general open books and closed braids, in some sense.

To state our main theorem, we first set up notations.
Let $(S,\phi)$ be an open book decomposition of a contact 3-manifold $(M,\xi)=(M_{(S,\phi)},\xi_{(S,\phi)})$ with respect to the Giroux correspondence \cite{gi}, and let $B$ be the binding. 
An oriented link $L$ in $M-B$ is a \emph{closed braid} (with respect to $(S,\phi)$), if $L$ is positively transverse to each page. The number of the intersections with $L$ and a page $S$ is denoted by $n(L)$ and called the \emph{braid index} of $L$.

By cutting $M$ along the page $S_{0}$, $L$ gives rise to an element $\alpha$ of $B_{n(L)}(S)$, the $n(L)$-strand braid group of the surface $S$. We say that $L$ is a \emph{closure} of $\alpha$, and denote by $L=\widehat{\alpha}$. Throughout the paper, we will fix a page $S_0$ and always see a closed braid as the closure of a braid.  

A closed braid is regarded as a transverse link in the contact 3-manifold $(M,\xi)$. For a null-homologous transverse link $L$ with Seifert surface $\Sigma$, we denote the self-linking number of $L$ with respect to $[\Sigma] \in H_{2}(M,L)$ by $sl(L,[\Sigma])$. To make notation simpler, we will omit to write $[\Sigma]$.

Apparently, the Jones-Kawamuro conjecture, even for original Jones' conjecture, fails for general open books and closed braids. Here is the simplest counter example.

\begin{example}
\label{exam:counter}
Let $(A,T_{A}^{-1})$ be an annulus open book with negative twist monodromy.
As we have seen in \cite[Example 2.20]{ik1-1}, there is a closed 1-braid $\widehat{\alpha}$ which is a transverse push-off of the boundary of an overtwisted disc (which we call a \emph{tranverse overtwisted disc}), so $sl(\widehat{\alpha})=1$.
On the other hand, the meridian of a connected component of the binding is a closed 1-braid $\widehat{\beta}$ with $sl(\widehat{\beta}) = -1$.
(See Example \ref{example:a} for further discussion).
\end{example}

Since this counter example comes from an overtwisted disc,
one may first hope that an open book supporting a \emph{tight} contact structure satisfies the inequality (\ref{eqn:JKconj2}).
However, as the next example due to Baykur, Etnyre, Van Horn-Morris and Kawamuro, shows this is not true, even for an open book decomposition of the standard contact $S^{3}$.

\begin{example}
\label{exam:counterBEHK}
Let $(A,T_{A})$ be an annulus open book with positive twist monodromy, and $\rho \in B_{1}(A) \cong \pi_{1}(A) \cong \Z$ be a generator of the 1-strand group of an  annulus that turns $A$ once in counter clockwise direction. The closed 1-braid $\widehat{\rho^{2}}$ is an unknot with $sl(\widehat{\rho^{2}}) = -3$. (See Example \ref{example:movie} for how to see this).
\end{example}

In fact, as we will discuss in Section \ref{sec:counterexamples}, almost all open books have closed braids violating the inequality (\ref{eqn:JKconj2}). Thus, to get a reasonable generalization of the Jones-Kawamuro conjecture, we need to add some assumptions and modify the statement.

The first assumption and modification we adopt is a topological one concerning closed braids. We concentrate our attention for the case that a knot can across only one particular component of the binding. Let us fix a connected component $C$ of the binding $B$, which we call the \emph{distinguished binding component}.
We say two links $L_1$ and $L_2$ in $M_{(S,\phi)}-B$ are \emph {$C$-topologically isotopic} if they are topologically isotopic in $M-(B-C) =(M-B)\cup C$.
We define the \emph{minimal $C$-braid index} of $L$ by
\[ b_{C}(L)= \min\{n(\widehat{\beta})\: | \: \widehat{\beta} \textrm{ is } C\textrm{-topologically isotopic to }L \}.\]
As we will see in Corollary \ref{cor:Markov}, two closed braids are $C$-topologically isotopic if and only if two closed braids are moved to the other by applying a sequence of braid isotopy and (de)stabilizations along the distinguished binding component $C$.

The second and the third assumptions we add concern the property of an open book. We consider the conditions
\begin{description}
\item[Planar] The page $S$ is planar.
\item[FDTC] The fractional Dehn twist coefficient (FDTC) along the distinguished binding $C$ satisfies $|c(\phi,C)| > 1$.
\end{description}

Here it is interesting to compare these two conditions with \cite[Corollary 1.2]{ik4} that states a planar open book $(S,\phi)$ with $c(\phi,C)>1$ for all $C \subset \partial S$ supports a tight contact structure.

Now our generalization of the Jones-Kawamuro conjecture is stated as follows:

\begin{theorem}[Generalization of the Jones-Kawamuro conjecture]
\label{theorem:main}
Let $(S,\phi)$ be an open book satisfying {\bf [Planar]} and {\bf [FDTC]} and $L \subset M_{(S,\phi)} -B$ be a null-homologous oriented link. If two closed braids $\widehat{\alpha}$ and $\widehat{\beta}$ are $C$-topologically isotopic to $L$, then the inequality 
\begin{equation}
\label{eqn:themain}
|sl(\widehat{\alpha})-sl(\widehat{\beta})| \leq 2( \max\{ n(\widehat{\alpha}),  n(\widehat{\beta}) \} - b_{C}(L))
\end{equation}
holds. 
\end{theorem}

\begin{remark}
For the case of the open book $(D^{2},\mathsf{Id})$, 
according to a convention $c(\mathsf{Id}_{D^{2}},\partial D^{2})=\infty$ explained in \cite{ik2} we may regard the open book $(D^{2},\mathsf{Id})$ satisfies {\bf [FDTC]}. In this case being $C$($=\partial D^{2}$)-topologically isotopic is equivalent to being topologically isotopic, so Theorem \ref{theorem:main} contains the Jones-Kawamuro conjecture (\ref{eqn:JKconj2}) as its special case. 
\end{remark}

Although the assumptions we add seem too restrictive at first glance, as we will see in Section \ref{sec:counterexamples}, Theorem \ref{theorem:main} is optimal in the sense that we cannot drop any assumptions from Theorem \ref{theorem:main}. 
We will present examples of closed braids $\widehat{\alpha}$ and $\widehat{\beta}$ in an open book $(S,\phi)$ violating the inequality (\ref{eqn:themain}), satisfying:
\begin{enumerate}
\item[(a)] $S$ is planar, $\widehat{\alpha}$ and $\widehat{\beta}$ are $C$-topologically isotopic, but $|c(\phi,C)|=1$ (Example \ref{example:a}).
\item[(b)] $S$ is planar, $|c(\phi,C)|>1$, and $\widehat{\alpha}$ and $\widehat{\beta}$ are topologically isotopic but are not $C$-topologically isotopic (Example \ref{example:b}).
\item[(c)] $\widehat{\alpha}$ and $\widehat{\beta}$ are $C$-topologically isotopic, and $|c(\phi,C)| >1$, but $S$ is not planar (Example \ref{example:c}).
\end{enumerate}

Our proof is inspired by LaFountain-Menasco's proof of the Jones-Kawamuro conjecture \cite{lm}, based on the braid foliation machinery developed by Birman and Menasco (see \cite{bf} for a basics of braid foliation). Among other things, foliation change and exchange move introduced in \cite{bm4,bm5}, and various observations and techniques developed in proving Markov Theorem Without Stabilization (MTWS) \cite{BM1,BM2} and usual Markov theorem \cite{BM0} play crucial roles. In our proof, we use an open book foliation machinery developed in \cite{ik1-1,ik2,ik3,ik4} which is a generalization of the braid foliation.

In Section \ref{sec:obf}, we review the open book foliation machinery, for pairwise disjoint annuli cobounded by two closed braids. We also summarize various operations on open book foliation which will be used in later.

In Section \ref{sec:top}, we prove that after suitable stabilizations of particular signs, topologically isotopic closed braids always cobound pairwise disjoint, embedded annuli. It should be emphasized that results in Section \ref{sec:top} hold for all open books and closed braids. As a corollary, we prove a slightly stronger version of the Markov theorem for closed braids in general open books in Corollary \ref{cor:Markov}, which is interesting in its own right.

In Section \ref{sec:proof} we prove Theorem \ref{theorem:main}. This is the point where we need to use assumptions {\bf [Planar]} and {\bf [FDTC]}, and the notion of $C$-topologically isotopic plays crucial roles.

In Section \ref{sec:c-circles}, we prove two Lemmas concerning the property of cobounding annuli with c-circles, which are used in the proof of Theorem \ref{theorem:main}. Existence of such cobounding annuli is a new feature of general open book foliation, which did not appear in braid foliation settings. 

In Section \ref{sec:counterexamples} we give various counter examples of the Jones-Kawamuro conjecture (\ref{eqn:themain}) for general open books to explain how our result is best-possible in a certain sense. In particular, in Proposition \ref{prop:counter}, we show that counter examples for a naive generalization of the  inequality (\ref{eqn:JKconj2}) are quite ubiquitous. This justifies our modification (\ref{eqn:themain}), a notion of $C$-topologically isotopic and the minimal $C$-braid index.

\section{Open book foliation machinery}
\label{sec:obf}

In this section we review open book foliation machinery which will be used in the proof of Theorem \ref{theorem:main}. For details, see \cite{ik1-1,ik2,ik3}. 

\subsection{Open book foliation for cobounding annuli}

Let $\widehat{\alpha}$ and $\widehat{\beta}$ be closed braids in $M_{(S,\phi)}$. Let $A$ be pairwise disjoint embedded annuli such that $\partial A =  \widehat{\alpha} \cup (-\widehat{\beta})$. We call such $A$ \emph{cobounding annuli between $\widehat{\alpha}$ and $\widehat{\beta}$}, and write $\widehat{\alpha} \sim_{A} \widehat{\beta}$.

In this section we review open book foliation machinery for cobounding annuli. Note that connected components $-\widehat{\beta}$ of $\partial A$ is \emph{negatively} transverse to pages. This gives rise to some new features in open book foliation, which we will briefly discuss.

Let us consider the the singular foliation $\F(A)$ on $A$ which is induced by intersections with pages
\[ \F(A)=\left\{ A \cap S_t \ | \ t \in [0, 1] \right\}. \]
We say that $A$ admits an open book foliation if $\F(A)$ satisfies the following conditions. 
\begin{description}
\item[($\mF$ i)] 
The binding $B$ pierces $A$ transversely in finitely many points. 
Moreover, for each $p \in B \cap A$ there exists a disc neighborhood $N_{p} \subset \Int(A)$ of $p$ on which the foliation $\F(N_p)$ is radial with the node $p$, see Figure~\ref{fig:sign}-(i). We call $p$ an {\em elliptic} point. 

\item[($\mF$ ii)] 
The leaves of $\F(A)$ are transverse to $\partial A$. 

\item[($\mF$ iii)] 
All but finitely many pages $S_{t}$ intersect $A$ transversely.
Each exceptional page is tangent to $A$ at a single point.
In particular, $\F(A)$ has no saddle-saddle connections.

\item[($\mF$ iv)] 
All the tangencies of $A$ and fibers are of saddle type, see Figure~\ref{fig:sign}-(ii). 
We call them {\em hyperbolic} points.

\end{description}

By isotopy fixing $\partial A$, $A$ can be put so that it admits an open book foliation (see \cite[Theorem 2.5]{ik1-1}). 
  
A leaf of $\F(A)$, a connected component of $A \cap S_t$ is {\it regular} if it does not contain a tangency point and is {\it singular} otherwise. We will often say that a hyperbolic point $h$ \emph{is around an elliptic point $v$}, if $v$ is an end point of the singular leaf that contains $h$.

The regular leaves are classified into the following four types:
\begin{enumerate}
\item[a-arc]: An arc where one of its endpoints lies on $B$ and the other lies on $\partial A$.
\item[b-arc]: An arc whose endpoints both lie on $B$.
\item[s-arc]: An arc whose endpoints both lie on $\partial A$.
\item[c-circle]: A simple closed curve.
\end{enumerate} 

By orientation reasons, an a-arc connects a positive elliptic point and a point of $\widehat{\alpha}$, or a negative elliptic point and a point of $\widehat{\beta}$.
Similarly, an s-arc connects a point of $\widehat{\alpha}$ and a point of $\widehat{\beta}$. A $b$-arc may connect different components of the binding.

\begin{figure}[htbp]
\begin{center}
\includegraphics*[bb=120 514 507 712,width=120mm]{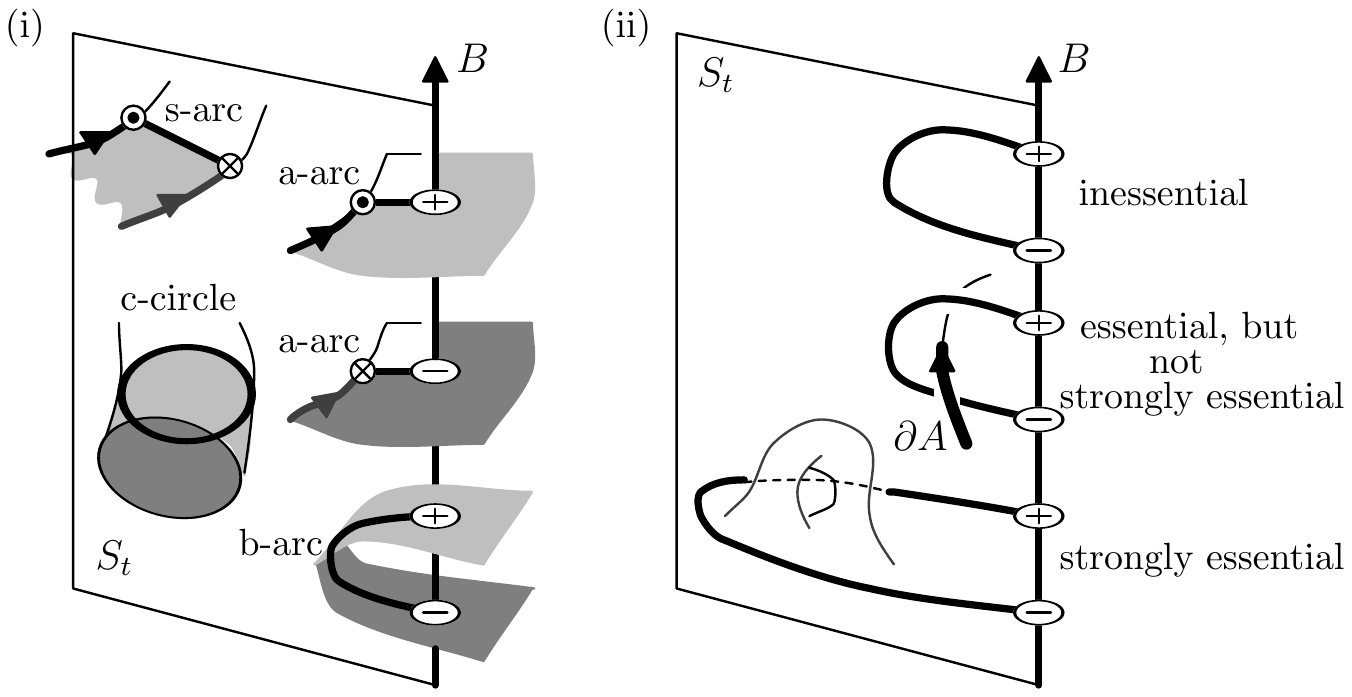}
\caption{(i) Regular leaves of open book foliation. (ii) essential and strongly essential b-arcs}
\label{fig:leaves}
\end{center}
\end{figure}

An elliptic point $p$ is {\em positive} (resp. {\em negative}) if the binding $B$ is positively (resp. negatively) transverse to $A$ at $p$.
The hyperbolic point $q$ is {\em positive} (resp. {\em negative}) if the positive normal direction $\vec{n}_{A}$ of $A$ at $q$ agrees (resp. disagrees) with the direction of the fibration. We denote the sign of a singular point $v$ by $\sgn(v)$.
See Figure \ref{fig:sign}.

\begin{figure}[htbp]
\begin{center}
\includegraphics*[bb= 163 551 442 715]{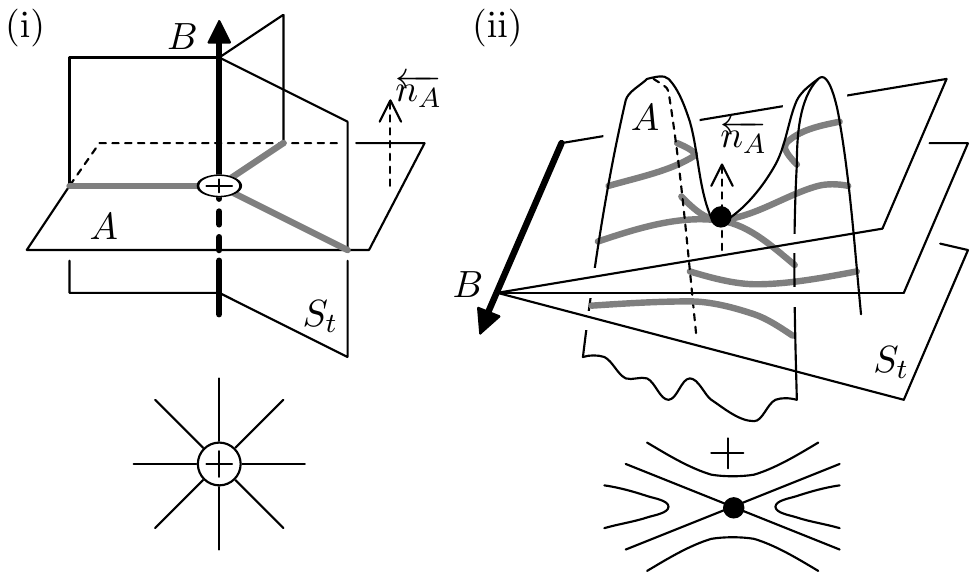}
\caption{Singular points and its signs for open book foliation: If the positive normal direction (illusrated by dotted arrow) of $A$ is opposite, we have a singular point with minus sign.}
\label{fig:sign}
\end{center}
\end{figure}

According to the types of nearby regular leaves, hyperbolic points are classified into nine types: Type $aa$, $ab$, $bb$, $ac$, $bc$, $cc$, $as$, $abs$, and $cs$. In the case of annuli, $ss$-singularity does not occur.
Each hyperbolic point has a canonical neighborhood as depicted in Figure ~\ref{fig:region}, which we call a {\em region}. We denote by $\sgn(R)$ the sign of the hyperbolic point contained in the region $R$. 
 
\begin{figure}[htbp]
\begin{center}
\includegraphics*[bb=166 502 439 708]{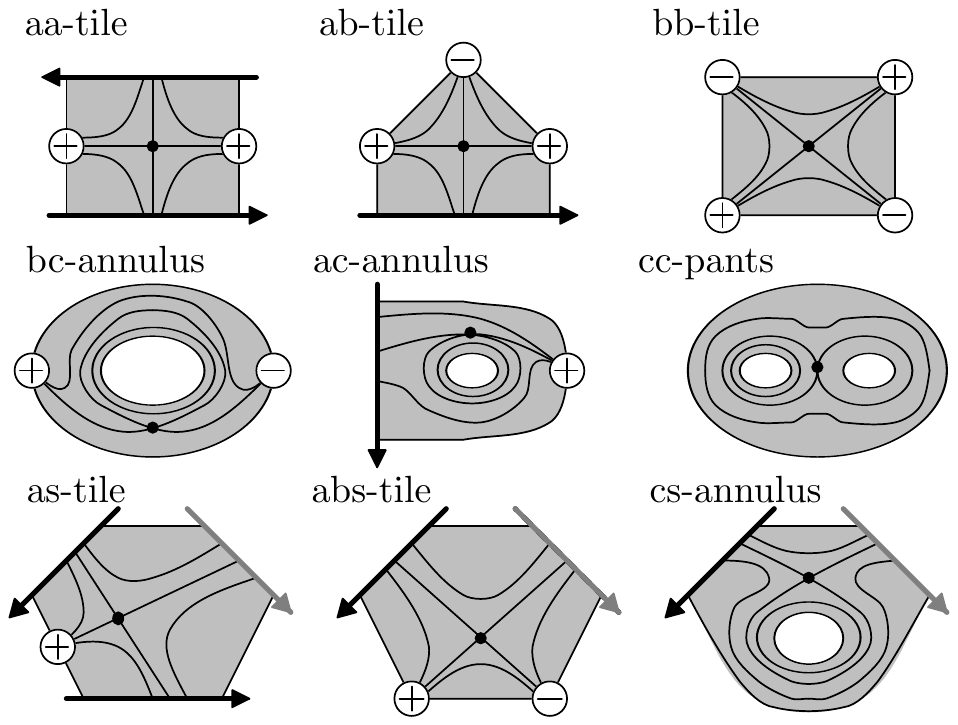}
\caption{Nine types of regions.}
\label{fig:region}
\end{center}
\end{figure}

If $\F(A)$ contains at least one hyperbolic point, then we can decompose $A$ as a union of regions whose interiors are disjoint \cite[Proposition 3.11]{ik1-1}. We call such a decomposition a \emph{region decomposition}.
In the region decomposition, some boundaries of a region $R$ can be identified. In such case, we say that $R$ is \emph{degenerated} (see Figure \ref{fig:degenerated}). Some degenerated region cannot exist, because around an elliptic point, all leaves must sit on distinct pages by {\bf ($\mathcal{F}$ i)}. 

\begin{figure}[htbp]
\begin{center}
\includegraphics*[bb=149 577 451 735,width=100mm]{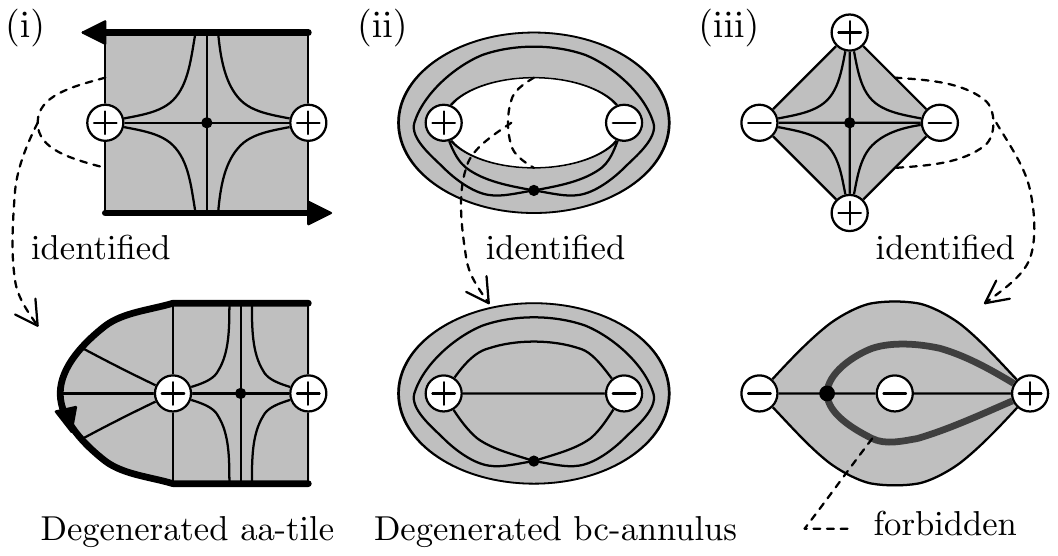}
\caption{Degenerated regions:  (iii) illustrates a forbidden degenenerated region. To see this is impossible, look at the leaf illustrated in bold-line. }
\label{fig:degenerated}
\end{center}
\end{figure}

A topological property of b-arc plays an important role. We say a b-arc $b \subset S_{t}$ is 
\begin{itemize}
\item {\em essential} if $b$ is not boundary-parallel in $S_{t} \setminus(S_{t} \cap \partial A)$.
\item {\em strongly essential} if $b$ is not boundary-parallel in $S_t$.
\item {\em separating} if $b$ separates the page $S_t$ into two components. 
\end{itemize} 
See Figure \ref{fig:leaves} (ii).

The conditions `{\em boundary parallel in $S_t$}' and `{\em non}-strongly essential' are equivalent. In this paper we prefer to use the former. Also note that a non-separating b-arc is always strongly essential. 
Finally we say that an elliptic point $v$ is {\em strongly essential} if every $b$-arc that ends at $v$ is strongly essential. 

For an element $\phi$ of the mapping class group of surface with boundary $S$ and   a connected component $C$ of $\partial S$, a rational number $c(\phi,C)$, called the \emph{fractional Dehn twist coefficients} (\emph{FDTC}, in short), is defined \cite{hkm1}. This number measures to what extent $\phi$ twists the boundary $C$, and plays an important role in contact geometry.

A key property of strongly essential elliptic point is that one can estimate the FDTC of the monodromy from such an elliptic point.

\begin{lemma}\cite[Lemma 5.1]{ik2}
\label{lemma:FDTC}
Let $v$ be an elliptic point of $\F(A)$ lying on a binding component $C \subset \partial S$. Assume that $v$ is strongly essential and there are no a-arc or s-arc around $v$. Let $p$ (resp. $n$) be the number of positive (resp. negative) hyperbolic points that lies around $v$. Then
\begin{enumerate}
\item If $\sgn(v)= +1$ then $-n \leq c(\phi,C) \leq p.$
\item If $\sgn(v)= -1$ then $-p \leq c(\phi,C) \leq n.$
\end{enumerate} 
\end{lemma}

We recall the following observation. 

\begin{proposition}\cite[Proposition 2.6]{ik1-1}
\label{prop:withoutc}
If cobounding annuli $A$ admit an open book foliation, then by ambient isotopy  fixing $\partial A$ we can put $A$ so that $\F(A)$ have no c-circles. Moreover, if the original cobounding annuli $A$ do not intersect with a component $C'$ of the binding $B$, then $\F(A)$ can be chosen so that no elliptic point of $\F(A)$ lie on $C'$.
\end{proposition}

We remark that when we put $A$ so that $\F(A)$ has no c-circles, in exchange, $\F(A)$ may have a lot of boundary-parallel b-arcs.

Finally, we remind the relation between the open book foliation and the self-linking number.

\begin{proposition}\cite[Proposition 3.2]{ik1-1}
\label{prop:slform}
Let $\Sigma$ be a Seifert surface of a closed braid $\widehat{\alpha}$, admitting an open book foliation. Then the self-linking number is given by
\[ sl(\widehat{\alpha},[\Sigma]) = -(e_{+}-e_{-})+(h_{+}-h_{-}),\]
where $e_{\pm}$ and $h_{\pm}$ the number of positive/negative hyperbolic points of the open book foliation of $\Sigma$.
\end{proposition}

\subsection{Movie presentation}

A movie presentation is a method to visualize an open book foliation of a surface $F$. See \cite[Section 2.1.5]{ik1-1} for details.

Let $F$ be an oriented surface embedded in $M_{(S,\phi)}$ so that it admits an open book foliation $\F(F)$.
We identify $\overline{ M_{(S,\phi)}-S_{0} }$ with $S\times[0,1]/\sim_{\partial}$, where $\sim_{\partial}$ is an equivalence relation given by $(x,t) \sim_{\partial} (x,s)$ for $x \in \partial S$ and $s,t \in [0,1]$.
Let $\mathcal{P}: \overline{M_{(S,\phi)}-S_{0}} \cong S\times[0,1]/\sim_{\partial} \rightarrow S$ be the projection given by $\mathcal{P}(x,t)=x$. We use $\mathcal{P}$ to fix the way of identification of the page $S_t$ with abstract surface $S$. In particular, when we draw the slice $(S_{t},S_{t}\cap F)$, we will actually  draw $\mathcal{P}(S_{t},S_{t}\cap F)$.

First we review a notion of describing arc for a hyperbolic point. By definition, a hyperbolic point $h$ is a saddle tangency of a singular page $S_{t^{*}}$ and $F$.
Let $N(h) \subset F$ be a saddle-shaped neighborhood of $h$. We put $F$ so that in the interval 
$[t^{*}-\varepsilon,t^{*}+\varepsilon]$ for a small $\varepsilon>0$, $F-N(h)$ is just a product. That is, the complement $F-N(h)$ is identified with $(S_{t^{*}} \cap (F-N(h)) \times [t^{*}-\varepsilon,t^{*}+\varepsilon]$.

The embedding of $N(h)$ is understood as follows: For $t \in [t^{*}-\varepsilon,t^{*})$, as $t$ increases two leaves $l_1(t)$ and $l_2(t)$ in $S_t$ approach along a properly embedded arc $\gamma \subset S_t$ joining $l_{1}$ and $l_{2}$, and at $t=t^{*}$ these two leaves collide to form a hyperbolic point. For $t \in (t^{*},t^{*}+\varepsilon]$, the configuration of leaves are changed (See Figure  \ref{fig:desarc}). Thus, the saddle $h$ is determined, up to isotopy, by an arc $\gamma \in S_{t^{*}-\varepsilon}$, which illustrates how two leaves $l_1(t)$ and $l_{2}(t)$ collide. We call $\gamma$ the \emph{describing arc} of the hyperbolic point $h$.

The describing arc also determines the sign of $h$: $\sgn(h)$ is positive (resp. negative) if and only if the positive normals $\vec n_F$ of $F$ pointing out of (resp. into) its describing arc.

\begin{figure}[htbp]
\begin{center}
\includegraphics*[bb=176 582 424 729,width=80mm]{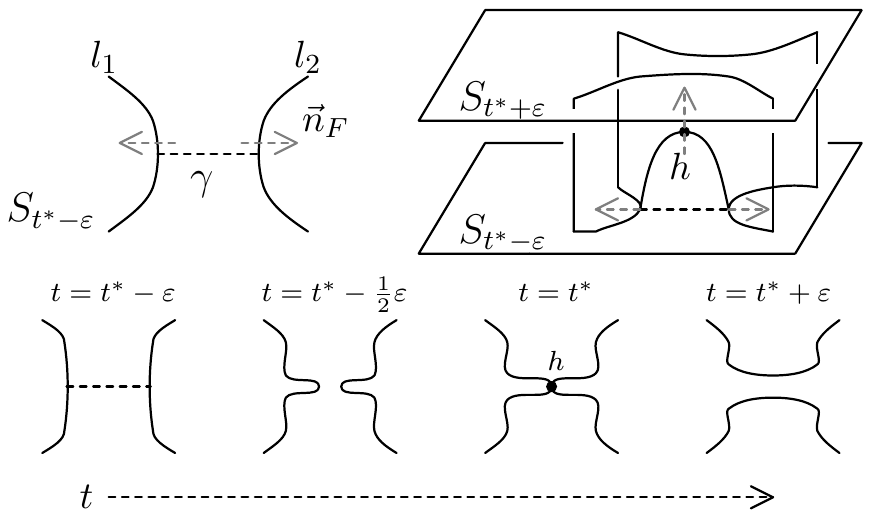}
\caption{Describing arc of hyperbolic point (for the case $\sgn = +$). We indicate the positive normal $\vec n_F$ by dotted gray arrows. We will illustrate the describing arc by dotted line.}
\label{fig:desarc}
\end{center}
\end{figure}

Take $0=s_{0}<s_{1}<\cdots < s_{k} =1$ so that $S_{s_{i}}$ is a regular page and that in each interval $(s_{i},s_{i+1})$ there exists exactly one hyperbolic point $h_i$. 
The sequence of slices $( S_{s_{i}}, S_{s_{i}} \cap F )$ with a describing arc of the hyperbolic point $h_i$ is called a \emph{movie presentation} of $F$. A movie presentation completely determines how the surface $F$ is embedded in $M_{(S,\phi)}$ and its open book foliation.
For convenience, to make it easier to chase how the surface and the braid move, we  often add redundant slices $(S_t,S_t \cap F)$ in the movie presentation.
  
\begin{example}[Movie for Example \ref{exam:counterBEHK}]
 \label{example:movie}

Here we give a movie presentation of the disc $D$ bounding the unknot $\widehat{\rho^{2}}$, in the open book $(A,T_{A})$ in Example \ref{exam:counterBEHK}. 

\begin{figure}[htbp]
\begin{center}
\includegraphics*[bb= 142 449 452 735,width=100mm]{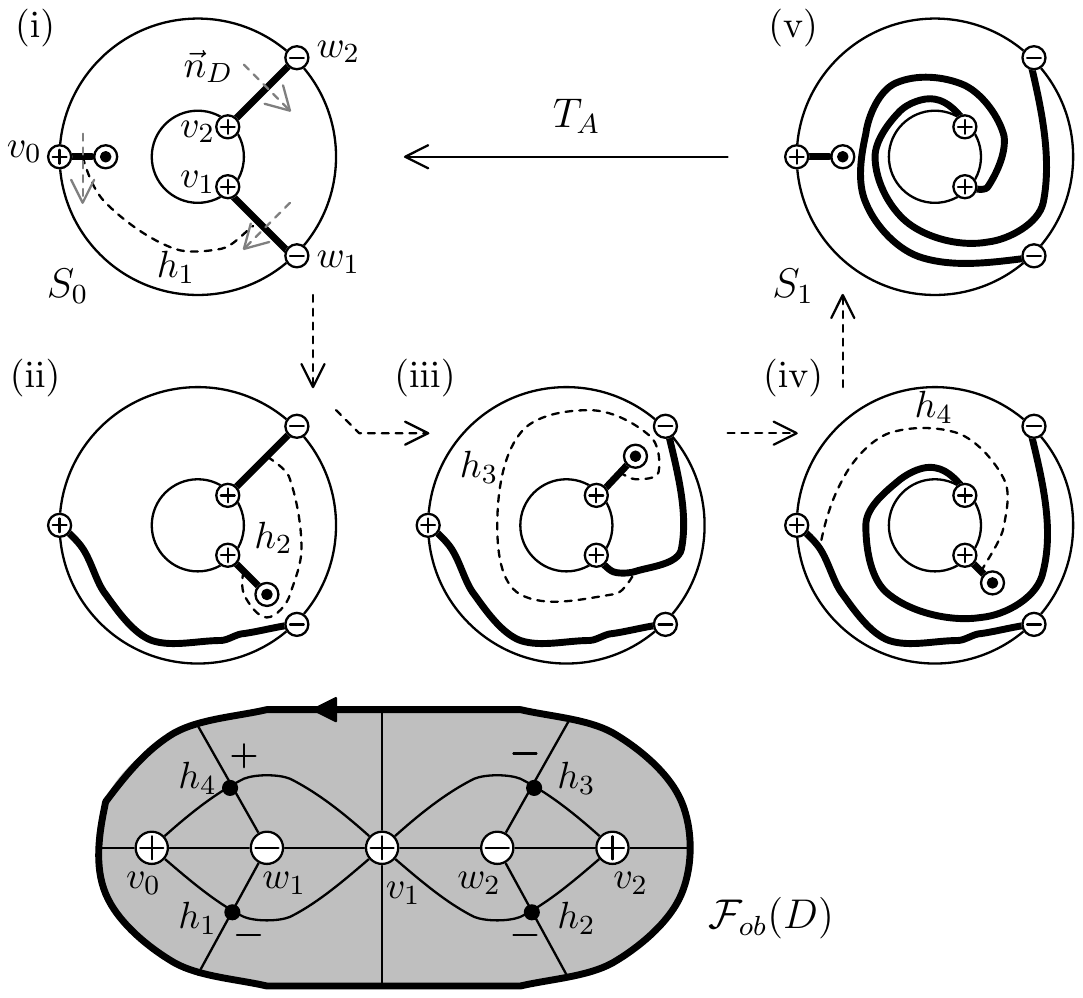}
\caption{Movie presentation of the disc $D$ bounding an unknot $\widehat{\rho^{2}}$ in the open book $(A,T_{A})$ (Example \ref{exam:counterBEHK})}
\label{fig:movieBEHK}
\end{center}
\end{figure}

\begin{enumerate}
\item[(i)] At $t=0$, we have one a-arc and two b-arcs. T positive normal $\vec n_{D}$ of $D$ is indicated by the gray, dotted arrow. As $t$ increases, the a-arc from $v_0$ and the b-arc connecting $v_1$ and $w_1$ forms a hyperbolic point $h_1$, whose describing arc is indicated by the dotted line. By the positive normal $\vec n_{D}$, the sign of $h_{1}$ is negative.
\item[(ii)] After passing the hyperbolic point $h_1$, we get an a-arc from $v_1$ and b-arc connecting $v_0$ and $w_1$. As $t$ increases, we then have a negative hyperbolic point $h_2$, indicated by the dotted line.
\item[(iii)] After passing the hyperbolic point $h_2$, we get an a-arc from $v_2$ and b-arc connecting $v_1$ and $w_2$. As $t$ increases, we then have a negative hyperbolic point $h_3$, indicated by the dotted line.
\item[(iv)] After passing the hyperbolic point $h_3$, we get an a-arc from $v_1$ and b-arc connecting $v_2$ and $w_2$. As $t$ increases, we then have a positive hyperbolic point $h_4$, indicated by the dotted line.
\item[(v)] After passing the hyperbolic point $h_4$, and at $t=1$ we have one a-arc  and two b-arcs. During the passage (i) -- (v), the boundary of a-arc winds twice in the annulus $A$. Finally, the slice at $t=1$ is mapped to the first slice (i) by the monodromy $T_{A}$ to give an embedded disc in $M_{(A,T_{A})}$.
\end{enumerate}

See Figure \ref{fig:movieBEHK}. From this movie presentation, we conclude $\F(D)$ is depicted as Figure \ref{fig:movieBEHK} so by Proposition \ref{prop:slform} we confirm that $sl(\widehat{\rho^{2}})=-3$, as asserted in Example \ref{exam:counterBEHK}.
\end{example}

\subsection{Review of operations on open book foliation}
\label{sec:operation}

In \cite{ik3}, we developed operations that modify the open book foliation. Such operations allow us to simplify the open book foliations and to put surfaces and closed braids in better positions.

These operations are realized by certain ambient isotopy which will often change the braid isotopy class and a position of surface dramatically, but when we just look at the open book foliation, they are local in the following sense:
For each operation there is a certain subset $U$ of $A$ such that the operation changes $\F(A)$ and the pattern of a region decomposition inside $U$, but it preserves $\F(A)$ outside of $U$.

Before describing operations on open book foliation, first we make it clear the meaning of stabilizations of closed braids. Let $C$ be a connected component of the binding $B$, and let $\mu_{C}$ be the meridian of $C$. We say a closed braid $\widehat{\alpha}$ is a positive (resp. negative) \emph{stabilization} of a closed braid $\widehat{\beta}$ along $C$, if $\widehat{\alpha}$ is obtained by connecting $\mu_C$ and $\widehat{\beta}$ along a positively (resp. negatively) twisted band.
Here a positively (resp. negatively) twisted band is a rectangle whose open book foliation has unique hyperbolic point with positive (resp. negative) sign. See Figure \ref{fig:stabilization}.

A positive stabilization preserves the transverse link types whereas a negative stabilization does not. If $\widehat{\alpha}$ is a negative stabilization of a closed braid $\widehat{\beta}$, $sl(\widehat{\alpha}) = sl(\widehat{\beta})-2$.

\begin{figure}[htbp]
\begin{center}
\includegraphics*[bb=163 632 434 733,width=90mm]{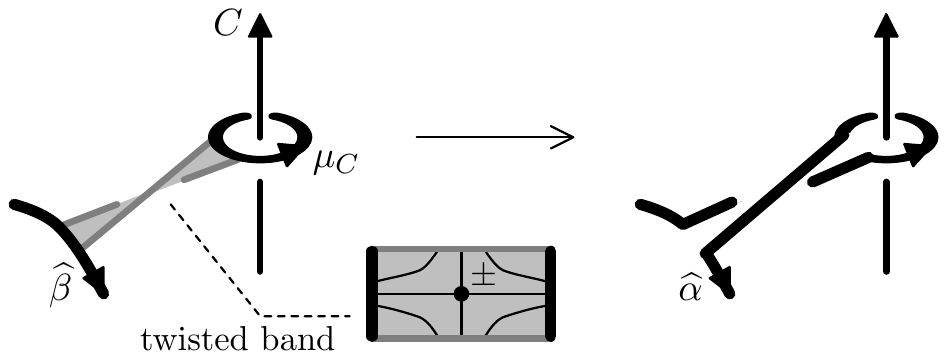}
\caption{Stabilization of closed braid}
\label{fig:stabilization}
\end{center}
\end{figure}

Now we summarize operations on open book foliation in somewhat casual way.
In Figure \ref{fig:operation} we illustrate five operations on open book foliations. The reader can understand these figures as a rule of changing the open book foliation, preserving the topological link types (or, the braid isotopy classes, or the transverse knot types) of $\partial A$. For detailed discussions and more precise statements, see \cite{ik3}.\\

\begin{figure}[htbp]
\begin{center}
\includegraphics*[bb=114 380 479 739,width=120mm]{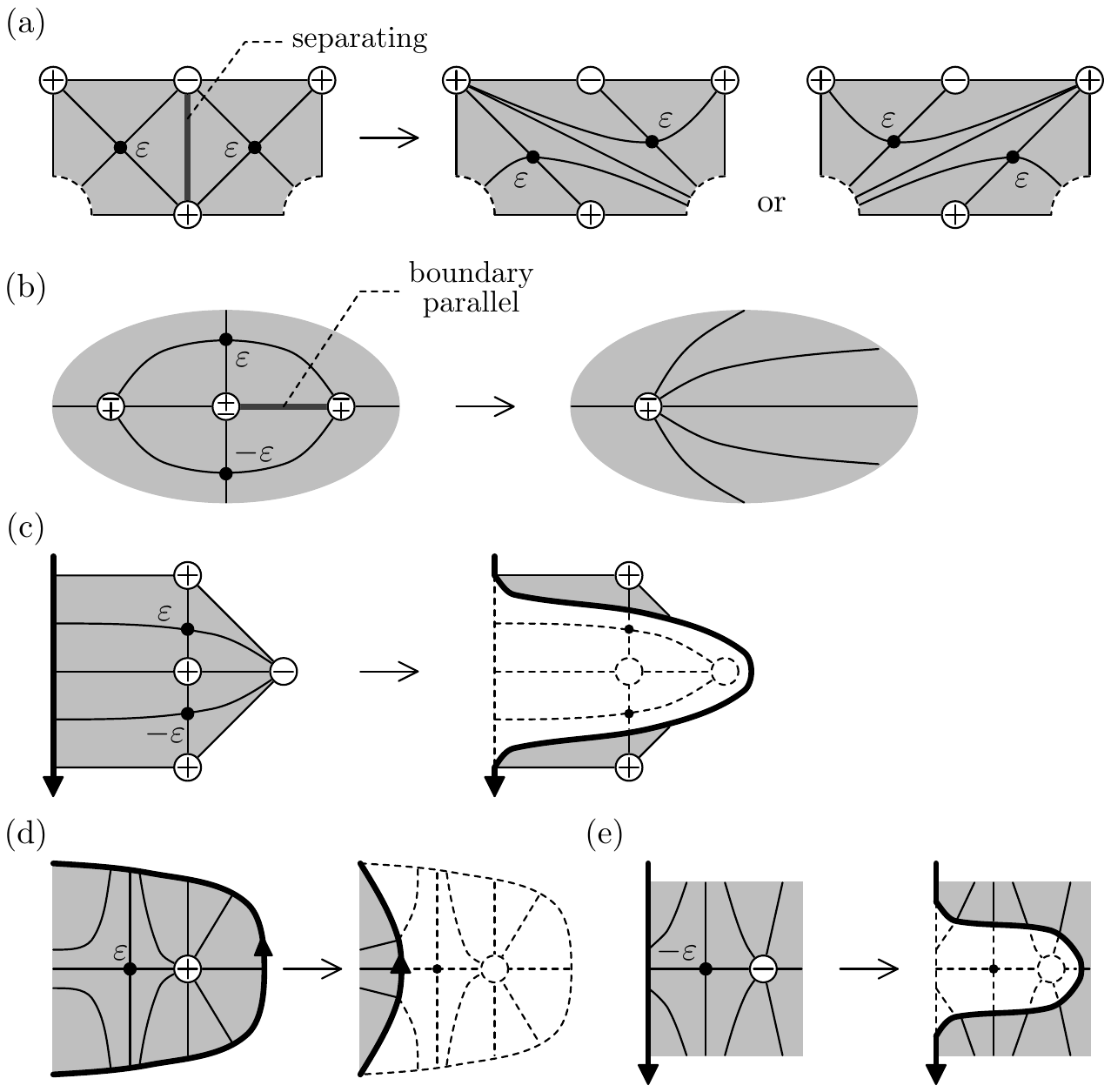}
\caption{Operations on open book foliations:
(a): b-arc foliation change, (b): interior exchange move, (c):boundary-shrinking exchange move, (d): destabilization (of sign $\varepsilon$), and (e): stabilization (of sign $\varepsilon$)}
\label{fig:operation}
\end{center}
\end{figure}

{\bf (a): b-arc foliation change}\\

The b-arc foliation change is an operation which changes the pattern of a region decomposition, designed to reduce the number of hyperbolic points around certain elliptic points. This operation preserves the braid isotopy class of $\partial A$.\\

Here is a precise setting. Assume that two ab- or bb- tiles $R_{1}$ and $R_{2}$ of the \emph{same} sign are adjacent at exactly one \emph{separating} b-arc $b$. Let $v_{\pm}$ be the positive and negative elliptic points which are the endpoints of $b$.
Then by ambient isotopy preserving the binding, one can change $R_{1} \cup R_{2}$ as a union of two new regions $R'_{1} \cup R'_{2}$ so that the number of hyperbolic points around $v_{\pm}$ decreases by one.\\

{\bf(b): Interior exchange move}\\

An interior exchange move, which was simply called an \emph{exchange move} in \cite{ik3}, is an operation that removes four singular points. This operation may change the braid isotopy class of $\partial A$, but preserves the transverse link types. \\

Assume that there exists an elliptic point $v$ contained in exactly two ab- or bb-tiles $R_{1}$ and $R_{2}$ of the \emph{opposite} signs, and that at least one of the common b-arc boundary $b$ of $R_1$ and $R_{2}$ is \emph{boundary parallel}. Then by ambient isotopy preserving the \emph{transverse link type} of $\partial A$ one can remove two hyperbolic points in $R_{1} \cup R_{2}$ and elliptic points which are the endpoints of $b$.\\

{\bf (c): Boundary-shrinking exchange move}\\

A boundary-shrinking exchange move is similar to the interior exchange move. Like interior exchange move, this operation may change the braid isotopy class of $\partial A$, but it preserves the transverse link type. A critical difference is that for a boundary-shrinking exchange move we do not require the common b-arc to be boundary-parallel. (This is the reason why we distinguish two exchange moves in a context of open book foliation.)\\ 

Assume that there exists an elliptic point $v$ contained in exactly two ab-tiles $R_{1}$ and $R_{2}$ of the \emph{opposite} signs. Then by ambient isotopy preserving the \emph{transverse link type} of $\partial A$, one can remove two regions $R_{1} \cup R_{2}$.\\

{\bf (d): Destabilization along a degenerated aa- or as- tile}\\

Let $R$ be a degenerated aa- or as-tile of sign $\varepsilon$, and $v$ is the positive elliptic point in $R$ which lies on a component $C$ of the binding $B$.
Then one can apply a destabilization of sign $\varepsilon$ along $C$ to remove the region $R$. In particular, the transverse link type of $\partial A$ is preserved if $\varepsilon = +$. \\

{\bf (e): Stabilization along an ab- or abs- tile }\\

Let $R$ be an ab- or abs- tile $R$ of sign $-\varepsilon$, and $v$ is the negative elliptic point in $R$ which lies on a component $C$ of the binding $B$. Then by applying stabilization of sign $\varepsilon$ along $C$, we can remove the region $R$. In particular, the transverse link type of $\partial A$ is preserved if $\varepsilon = -$. \\

Since a boundary shrinking exchange move is not discussed in \cite{ik3}, we give a concise explanation. The boundary shrinking exchange move is as a composite of the stabilization along an ab-tile {\bf (e)} and the destabilization along a degenerated aa-tile {\bf (d)}, as shown in Figure \ref{fig:bexmove}. 
The condition $\sgn (R_{1}) \neq \sgn(R_{2})$ guarantees that we are able to choose the signs of stabilizations and destabilizations are positive so the boundary shrinking exchange move preserves the transverse link type.
 
\begin{figure}[htbp]
\begin{center}
\includegraphics*[bb=125 645 471 729, width=120mm]{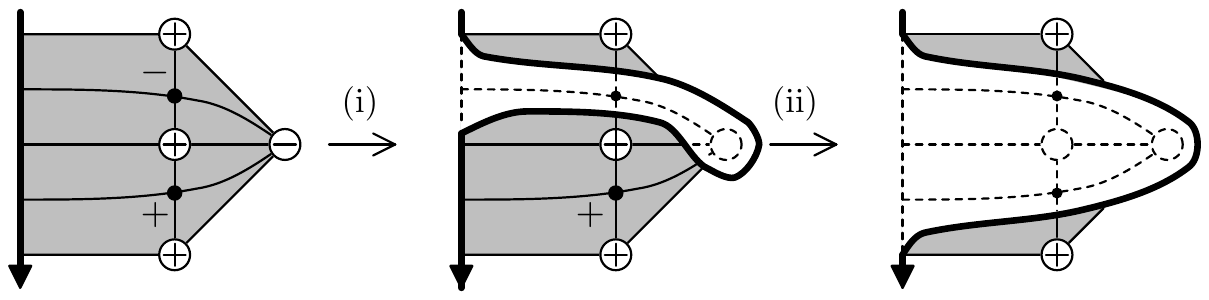}
\caption{Boundary-shrinking exchange move is realized by positive stabilization (i) and positive destabilization (ii).}
\label{fig:bexmove}
\end{center}
\end{figure}

In a 3-dimensional picture, boundary shrinking exchange move can be understood as a move sliding the braid along a part of surface $R_1 \cup R_2$ which forms a ``pocket''. See Figure \ref{fig:bex}.

\begin{figure}[htbp]
\begin{center}
\includegraphics*[bb=107 609 495 732, width=130mm]{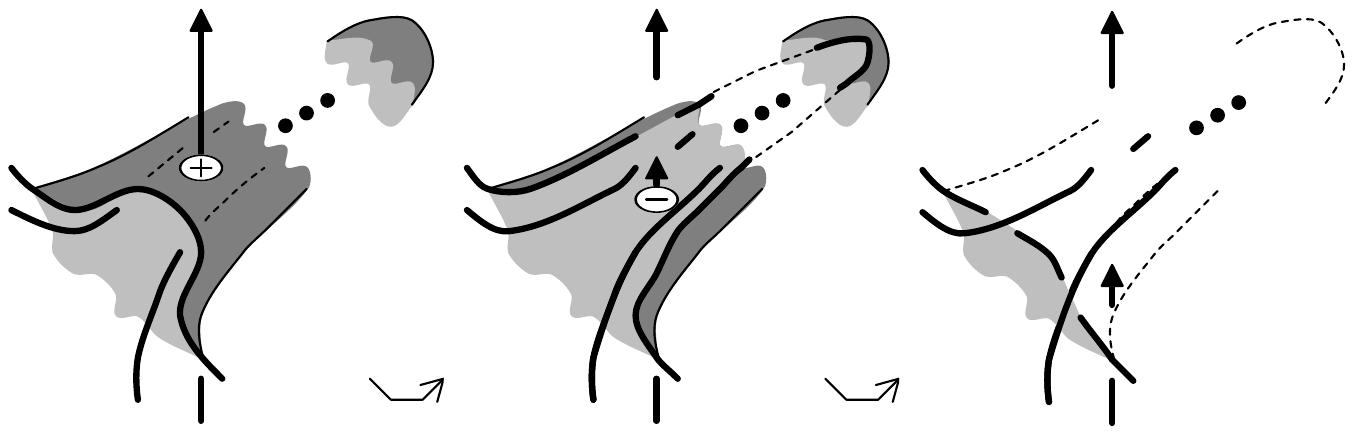}
\caption{Isotopy realizing a boundary shrinking exchange move}
\label{fig:bex}
\end{center}
\end{figure}

In the rest of the paper, we will often simply call an \emph{exchange move} to mean both an interior exchange move and a boundary-shrinking exchange move, if their differences are not important. Also, we say an exchange move is \emph{along $C$} if two elliptic points which will be removed by the move lie on $C$. In such case, the original closed braid and the resulting closed braid after exchange move are $C$-topologically isotopic.

\section{Topologically isotopic closed braids stably cobound annuli}
\label{sec:top}

In this section, we prove a generalization of \cite[Proposition 1.1]{lm}, which asserts that every topologically isotopic closed braids $\widehat{\alpha}$ and $\widehat{\beta}$ in $S^{3}$ cobound embedded annuli, after positively stabilizing $\widehat{\alpha}$ and negatively stabilizing $\widehat{\beta}$. 
This results can be generalized to arbitrary open books and closed braids, without any additional assumptions.

\begin{theorem}
\label{theorem:stablecobound}
If two closed braids $\widehat{\alpha}$ and $\widehat{\beta}$ in an open book $(S,\phi)$ are topologically isotopic, then there exist closed braids $\widehat{\alpha_{+}}$ and $\widehat{\beta_{-}}$ which are positive stabilizations of $\widehat{\alpha}$ and negative stabilizations of $\widehat{\beta}$ respectively, such that $\widehat{\alpha_{+}}$ and $\widehat{\beta_{-}}$ cobound pairwise disjoint embedded annuli $A$.

Moreover, if $\widehat{\alpha}$ and $\widehat{\beta}$ are $C$-topologically isotopic for a distinguished binding component $C$, then 
all stabilizations are stabilizations along $C$, and the cobounding annuli $A$ can be chosen so that they do not intersect with the rest of the binding components, $B-C$.
\end{theorem}

\begin{proof}

First we take a sequence of closed braids and cobounding annuli 
\begin{equation}
\label{eqn:seqproof}
\widehat{\alpha} \sim_{A} \widehat{\alpha_1} \sim_{A_1} \widehat{\alpha_2} \sim_{A_2} \cdots  \sim_{A_{k-1}} \widehat{\alpha_k} \sim_{A_k} \widehat{\beta}
\end{equation}
so that the property 
\begin{enumerate}
\item[($\ast$)] $\widehat{\alpha}$ intersects the $i$-th cobounding annuli $A_i$ with at most one point for each $i\geq 1$
\end{enumerate}
holds.

Such a sequence of cobounding annuli and closed braids are obtained as follows:
Since $\widehat{\alpha}$ and $\widehat{\beta}$ are topologically isotopic, there exists a sequence of links which may not be closed braids,
\begin{equation}
\label{eqn:1stseq} 
\widehat{\alpha} = L_0 \rightarrow L_1 \rightarrow \cdots \rightarrow L_{k-1} \rightarrow L_{k} = \widehat{\beta} 
\end{equation}
such that $L_i \cup (-L_{i+1})$ cobound pairwise disjoint embedded annuli $A'_i$ in $M_{(S,\phi)}$.
By subdividing the sequence (\ref{eqn:1stseq}), we may assume that each $A'_{i}$ intersects $\widehat{\alpha}=L_{0}$ with at most one point.

We inductively modify the sequence (\ref{eqn:1stseq}) to produce the desired sequence of closed braids and cobounding annuli. First we put $\widehat{\alpha_0}= \widehat{\alpha}$ and $A''_0 =A'_0$.

Assume that we have obtained a sequence of links, closed braids and cobounding annuli 
\[ \widehat{\alpha} = \widehat{\alpha_0} \sim_{A} \widehat{\alpha_1} \sim_{A_1} \cdots \sim_{A_{i-1}} \widehat{\alpha_{i-1}} \rightarrow L_{i} \rightarrow L_{i+1} \rightarrow \cdots \rightarrow L_k\]
so that the property ($\ast$) holds and that $\widehat{\alpha_{i-1}} \cup (-L_{i})$ cobound annuli $A''_{i}$ that intersect $\widehat{\alpha}$ with at most one point.

We apply Alexander's trick to $L_{i}$ to get a closed braid $\widehat{\alpha_{i}}$ as follows. With no loss of generality, we may assume that $L_{i}$ is transverse to pages except finitely many points. Assume that some portion $\gamma$ of $L_{i}$ is negatively transverse to pages. 
Then we take a disc $\Delta$ with properties 
\begin{enumerate}
\item The boundary $\partial \Delta$ is a closed 1-braid that is decomposed as union of two arcs, $\partial \Delta = (-\gamma) \cup \gamma'$ and $\Delta \cap L_i = \gamma$.
\item The disc $\Delta$ is positively transverse to the binding at one point. Moreover, the intersection $\Delta \cap B$ lies on  the distinguished binding component $C$, if necessary. 
\item The disc $\Delta$ is disjoint from $\widehat{\alpha} \cup \widehat{\alpha_{i-1}} \cup L_{i+1}$.
\item An interior of the regular neighborhood $N(\gamma)$ of $\gamma$ in the disc $\Delta$ is disjoint from both $A''_{i}$ and $A'_{i+1}$.
\end{enumerate}

Then we replace the link $L_i$ with a new link $(L_i -\gamma)\cup \gamma'$. This removes the negatively transverse portion $\gamma$ of $L_{i}$.
Moreover, by property (3) and (4) above, by attaching $\Delta$ to $A''_i$ or $A'_{i+1}$, we extend the cobounding annuli. Here, if $\Delta$ intersects with the cobounding annuli $A'= A''_i$ or $A'_{i+1}$, we push $A'$ along $\Delta$ to make them disjoint from $\Delta$, as we illustrate in Figure \ref{fig:Atrick} (By (3), we may assume that other types of intersections does not appear). Since $\Delta$ is chosen to be disjoint from $\widehat{\alpha}$, this modification does not produce new intersections with $\widehat{\alpha}$. In particular, the resulting cobounding annuli preserves the property that they intersect $\widehat{\alpha}$ with at most one point.
  
After applying this operation (which we call Alexander's trick) finitely many times, we modify $L_i$ so that it is a closed braid $\widehat{\alpha_i}$, and obtain the cobounding annuli $A_i$ between $\widehat{\alpha_{i-1}}$ and $\widehat{\alpha_i}$, and the cobounding annuli $A''_{i+1}$ between $\widehat{\alpha_i}$ and $L_{i+1}$ as desired. 
Note that if $\widehat{\alpha}$ and $\widehat{\beta}$ are $C$-topologically isotopic, then one can choose the cobounding annuli $A'_i$ so that they are disjoint from $B-C$. Hence we can take all cobounding annuli $A_i$ so that they are disjoint from $B-C$.

\begin{figure}[htbp]
\begin{center}
\includegraphics*[bb=153 569 452 708,width=90mm]{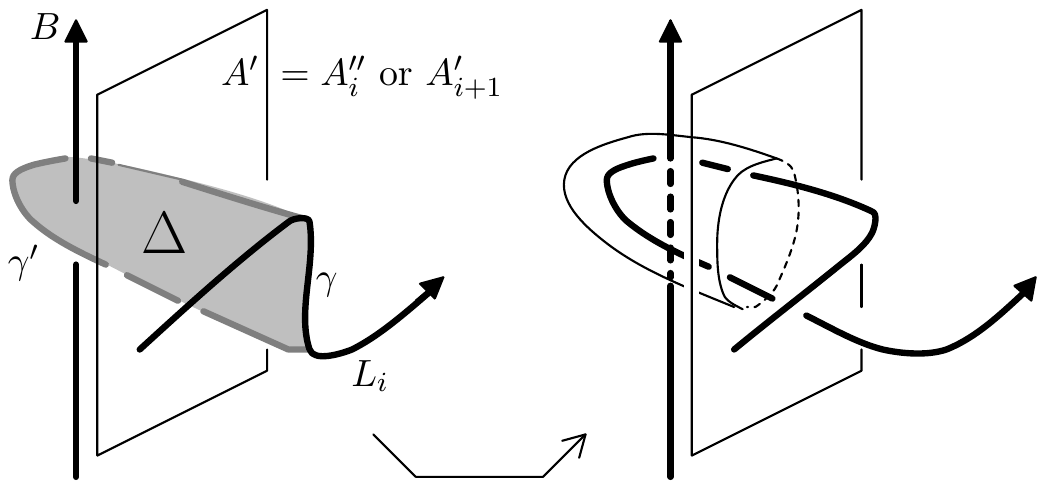}
\caption{Alexander's trick:}
\label{fig:Atrick}
\end{center}
\end{figure}

Now we use a sequence (\ref{eqn:seqproof}) to prove the theorem.
We show that by shrinking the first cobounding annuli $A_1$ appropriately, we obtain a new sequence of closed braids and cobounding annuli with shorter length
\begin{equation}
\label{eqn:newseq}
\widehat{\alpha_+} \sim_{A^*} \widehat{\alpha_{2-}} \sim_{A^*_2}\widehat{\alpha_3} \sim_{A_3}  \cdots \sim_{A_{k-1}} \widehat{\alpha_k} \sim_{A_k} \widehat{\beta}
\end{equation}
where $\widehat{\alpha_+}$ and $\widehat{\alpha_{2-}}$ are the positive and negative stabilizations of $\widehat{\alpha}$ and $\widehat{\alpha_{2}}$ respectively, and the new cobounding annuli $A^{*}, A_2^{*},A_{3},\ldots$ satisfy the property corresponding to ($\ast$),
\begin{enumerate}
\item[($\ast$)] $\widehat{\alpha_{+}}$ intersects the cobounding annuli $A^{*}$, $A_{2}^{*},A_{3},\ldots$ with at most one point.
\end{enumerate}
 Once this is done, an induction on the length of the sequence (\ref{eqn:seqproof}) of cobounding annuli proves the theorem.

In the rest of the proof, we give a construction of shorter sequence (\ref{eqn:newseq}). By Proposition \ref{prop:withoutc}, we can put $A_1$ so that it admits an open book foliation without c-circles. We modify and shrink the annuli $A_1$ in the following five steps. See Figure \ref{fig:summary} for an overview of our construction.\\

\begin{figure}[htbp]
\begin{center}
\includegraphics*[bb= 128 484 527 713,width=140mm]{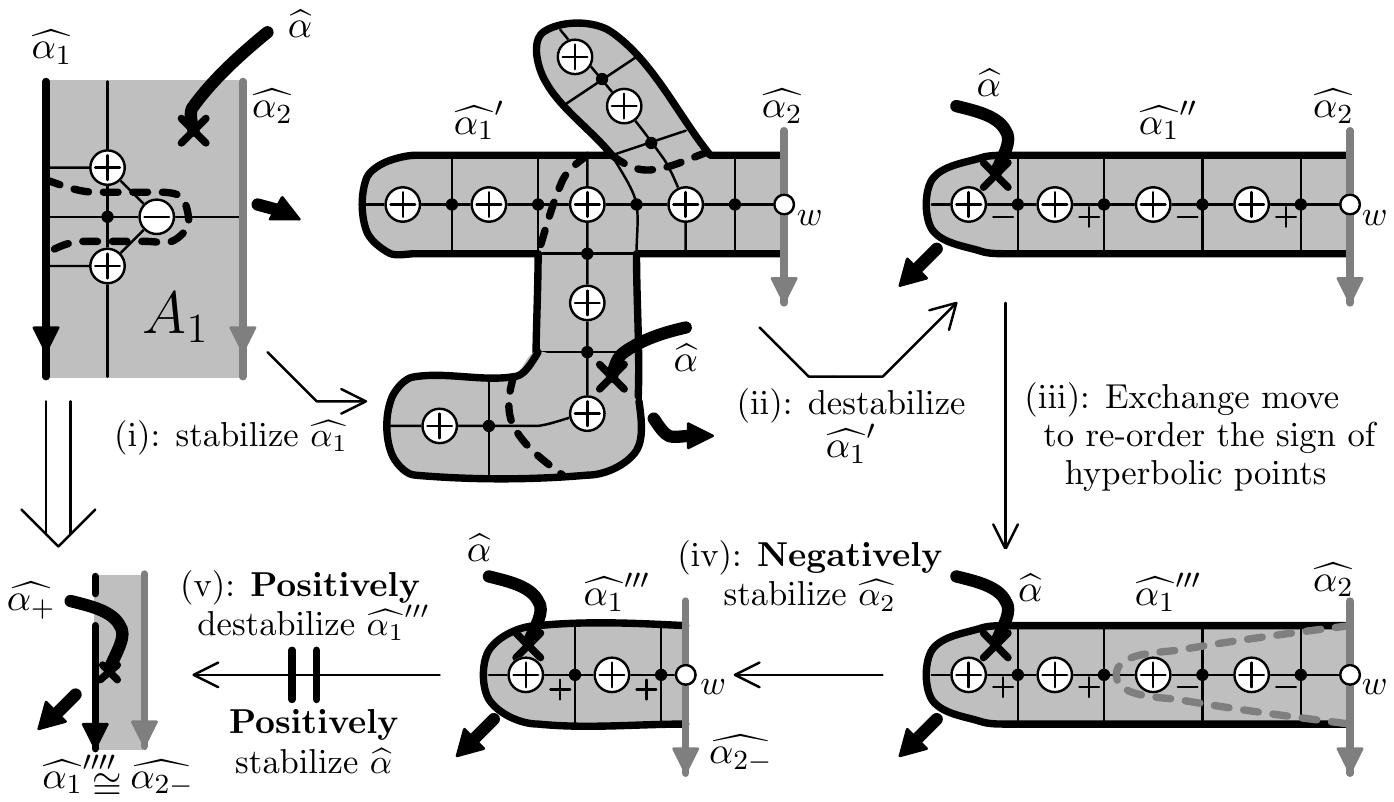}
\caption{Summary: how to get new sequence of cobounding annuli and closed braids $\widehat{\alpha_+}$ and $\widehat{\alpha_{2-}}$.}
\label{fig:summary}
\end{center}
\end{figure}

{\bf (i): Removing negative elliptic points (stabilizations for $\widehat{\alpha_1}$) }\\

In the first two steps, we do not care about the sign of hyperbolic points.
We stabilize $\widehat{\alpha_1}$ along ab- or abs- tiles of $A_1$ (see Section \ref{sec:operation} {\bf (d)}) to remove all negative elliptic points from $\F(A_1)$. We denote the resulting closed braid by $\widehat{\alpha_1}'$.
The cobounding annuli $A$ yield the cobounding annuli $A'$ between $\widehat{\alpha}$ and $\widehat{\alpha_1}'$.
\\

{\bf (ii) Removing degenerated aa-tiles (destabilizations for $\widehat{\alpha_1}'$)}\\

After the step (i), the region decomposition of $\F(A_1)$ consists of only aa- and as-tiles and $A_1$ is a union of discs $D_1,\ldots,D_m$ and strips foliated by s-arcs. We may assume that the intersection point $A_1 \cap \widehat{\alpha}$ lies on $D_1$. 

For each $D_i$, let us consider the tree $G_i$ whose vertices are elliptic points and a point $w$ on $\widehat{\alpha_2}$ which is the end point of a singular leaf, an  whose edges are singular leaves connecting vertices. 

Except $w$, a vertex of valence one in the tree $G_i$ is nothing but an elliptic point contained in a degenerated aa- or as-tile $R$. 
If $R$ does not intersect with $\widehat{\alpha}$, by destabilizing $\widehat{\alpha_1}'$ along $R$ (see Section \ref{sec:operation} {\bf (e)}), we remove $R$ to simplify the disc $D_i$, without affecting $\widehat{\alpha}$.

Since $D_i$ does not intersect with $\widehat{\alpha}$ for $i>1$, by destabilizations we eventually remove $D_i$.
For the disc $D_1$, we destabilize $\widehat{\alpha_1}$ until the unique intersection point $D_1 \cap \widehat{\alpha}$ obstructs.
Let us write the resulting closed braid by $\widehat{\alpha_1}''$. Again, the cobouding annuli $A'$ give the cobounding annuli $A''$ between $\widehat{\alpha}$ and $\widehat{\alpha_1}''$.
\\

{\bf (iii): Re-ordering the sign of hyperbolic points (exchange moves for $\widehat{\alpha_1}''$)}\\

From now on, we carefully look at the sign of hyperbolic points.
After the step (ii), $\F(A_1)$ is a union of a strip foliated by s-arcs and the disc $D_1$. The graph $G_1$ is a linear graph and $D_1$ is a linear string of as- and aa-tiles. We re-order the sign of hyperbolic points in $D_1$ as follows.

Let us consider the situation that there is a positive elliptic point $v$ such that $v$ is contained in two aa-tiles with opposite signs (see Figure \ref{fig:swap}). Let $C$ be the connected component of $B$ on which $v$ lies, and let $N(v) \cong D^{2} \times [-1,1] \subset D^{2}\times C$ be the regular neighborhood of $v$ in $M$. By suitable ambient isotopy, we put two aa-tiles so that their hyperbolic points are contained in $N(v)$. In a ball $N(v)$, we apply the classical exchange move to exchange the over and under strands (this notion make sense, by considering the projection $N(v) \cong D^{2} \times [-1,1] \to D^{2}$). As a consequence, the sign of the two hyperbolic points are swapped.

\begin{figure}[htbp]
\begin{center}
\includegraphics*[bb= 182 468 429 705,width=80mm]{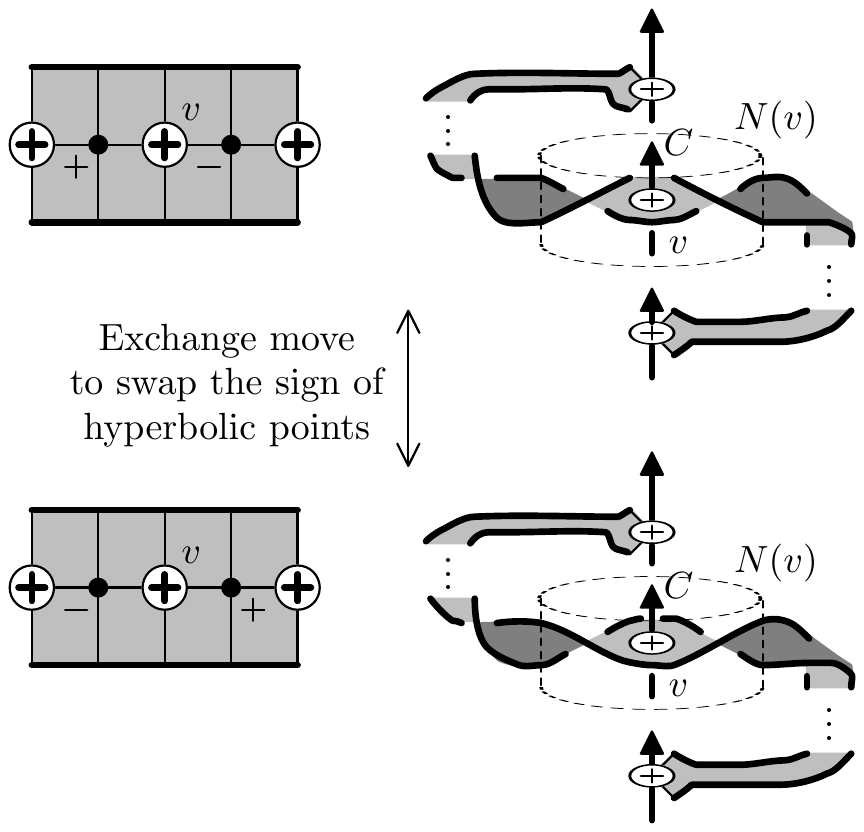}
\caption{Exchange move to swap the sign of adjacent two aa-tiles}
\label{fig:swap}
\end{center}
\end{figure}

Thus, by applying exchange moves for $\widehat{\alpha}''$, we can arrange the sign of hyperbolic points in $D_1$ so that the negative hyperbolic points are compiled to the end $w$ and the positive hyperbolic points are compiled to the other end. We denote the resulting closed braid by $\widehat{\alpha_1}'''$. 
The cobounding annuli $A''$ produce the cobounding annuli $A'''$ between $\widehat{\alpha}$ and $\widehat{\alpha_1}'''$.
\\

{\bf (iv) Removing negative hyperbolic points (negative stabilizations for $\widehat{\alpha_2}$)}\\

After the step (iii), the sign of the as-tile is negative unless the sign of all hyperbolic points in $D_1$ are positive. We negatively stabilize $\widehat{\alpha_2}$ along a negative as-tiles, until the sign of as-tile become positive (see Section \ref{sec:operation} {\bf (d)}). As a consequence, we remove all negative hyperbolic points from $\F(A)$ and we get a new closed braid $\widehat{\alpha_{2-}}$, a negative stabilizations of $\widehat{\alpha_2}$.

The cobounding annuli $A_2$ between $\widehat{\alpha_2}$ and $\widehat{\alpha_3}$ produce the cobouning annuli $A^{*}_2$ between $\widehat{\alpha_{2-}}$ and $\widehat{\alpha_3}$.
 In the construction of $A^{*}_2$, new intersection with $\widehat{\alpha}$ is never created, so $\widehat{\alpha}$ intersects the new cobounding annuli $A^{*}_2$ with at most one point.\\

{\bf (v) Removing positive hyperbolic points (positive stabilizations for $\widehat{\alpha}$)}\\

After the step (iv), all hyperbolic points in $D_1$ are positive.
In the last step, we positively destabilize $\widehat{\alpha_1}$ to shrink the rest of the cobounding annuli $A_1$. Since the degenerated aa-tile in $D_1$ intersects $\widehat{\alpha}$, the destabilization along the degenerated aa-tile causes a change of the closed braid $\widehat{\alpha}$. The change of  $\widehat{\alpha}$ induced by the destabilization is understood as follows. 

Let $v$ be the positive elliptic point in the degenerated aa-tile $R$ and $C$ be the connected point on which $v$ lies. As in the step (iv), let $N(v) \cong D^{2} \times [-1,1] \subset D^{2}\times C$ be the regular neighborhood of $v$ in $M_{(S,\phi)}$. We may assume that $R$ is contained in $N(v)$, and that the cobounding annuli $A_3,\ldots$ does not intersect with $N(v)$, to guarantee that the change of $\widehat{\alpha}$ does not create new intersection points.

As is noted in \cite{lm}, for a link $\widehat{\alpha} \cup \widehat{\alpha_1}'''$,  positive destabilization $\widehat{\alpha_1}'''$ induces a move which is called the  microflype, the simplest flype move in braid foliation theory. (see \cite[Section 2.3, Section 5.3]{BM1}). 
A positive destabilization $\widehat{\alpha_1}'''$ leads to a positive stabilization of $\widehat{\alpha}$. This isotopy of link $\widehat{\alpha} \cup \widehat{\alpha_1}'''$ is supported in $N(v)$.

\begin{figure}[htbp]
\begin{center}
\includegraphics*[bb= 179 593 440 707,width=90mm]{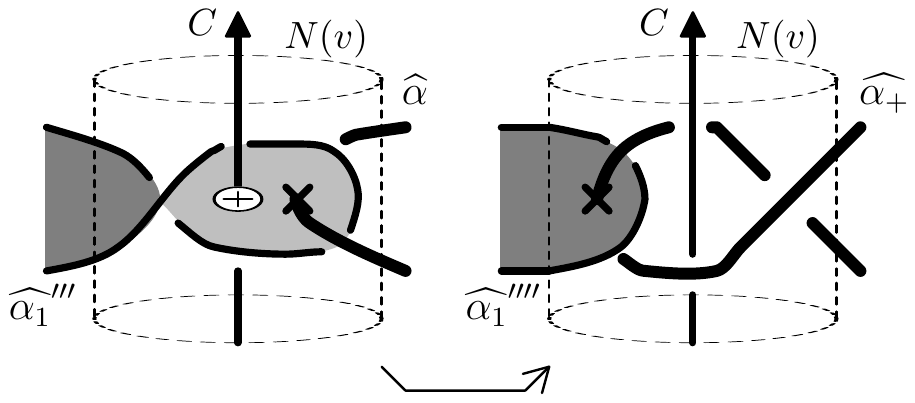}
\caption{Microflype: a positive destabilization of $\widehat{\alpha_1}'''$ induces a positive stabilization of $\widehat{\alpha}$.}
\label{fig:microflype}
\end{center}
\end{figure}

Therefore after applying microflypes which induces positive destabilizations for $ \widehat{\alpha_1}'''$ and positive stabilizations for $\widehat{\alpha}$, we eventually remove all singular points, so $A_1$ is now foliated by s-arcs.
Let us call the resulting closed braids $\widehat{\alpha_1}''''$ and  $\widehat{\alpha_+}'$, respectively. The cobounding annuli $A'''$ give the cobounding annuli $A''''$ between $\widehat{\alpha_+}'$ and $\widehat{\alpha_1}''''$.

Let $N(A_1)$ be a regular neighborhood of $A_1$ in $M_{(S,\phi)}$. 
We put each cobounding annuli $A^{*}_2$, $A_{3},\ldots$ so that the intersection with $\widehat{\alpha_+}'$ does not lie in $N(A_1)$.

Since $A_1$ is foliated by s-arcs, there is an ambient isotopy $\Phi_t: M_{(S,\phi)} \rightarrow M_{(S,\phi)}$ such that:
\begin{itemize}
\item $\Phi_t$ preserves each page of open book.
\item $\Phi_0= \textsf{id}$ and $\Phi_1(\widehat{\alpha''''}) = \widehat{\alpha_{2-}}$.
\item $\Phi_t = \textsf{id}$ outside $N(A_1)$.
\end{itemize}

Let $\widehat{\alpha_{+}} = \Phi_1(\widehat{\alpha_{+}}')$ and $A^{*} = \Phi_{1}(A'''')$. Then $A^{*}$ is a cobounding annuli between $\widehat{\alpha_{+}}$ and $\widehat{\alpha_{2-}}$. Moreover, $\widehat{\alpha_{+}}$ can intersect with each  cobounding annuli $A^{*}_2$, $A_{3},\ldots$ at most one point.
This completes the construction of new sequence (\ref{eqn:newseq}).

\end{proof}

We noticed that Theorem \ref{theorem:stablecobound} shows the Markov theorem for general open book, in a slightly stronger form than stated in \cite{sk}:

\begin{corollary}[Markov Theorem for general open book]
\label{cor:Markov}
If two closed braids $\widehat{\alpha}$ and $\widehat{\beta}$ are topologically isotopic, then they admit a common stabilization:
namely, there exists a sequence of closed braids
\[ \widehat{\alpha} = \widehat{\alpha_0}\rightarrow \widehat{\alpha_1} \rightarrow \cdots \rightarrow \widehat{\alpha_{k}} \cong \widehat{\beta_l} \leftarrow \cdots \leftarrow \widehat{\beta_1}  \leftarrow \widehat{\beta_0}=\widehat{\beta} \]
such that $\widehat{\alpha_{i+1}}$ (resp. $\widehat{\beta_{j+1}}$) is obtained from $\widehat{\alpha_i}$ (resp. $\widehat{\beta_j}$) by a stabilization or a braid isotopy.

Moreover, if $\widehat{\alpha}$ and $\widehat{\beta}$ are $C$-topologically isotopic for some component of the binding $C$, then all stabilizations are chosen to be a stabilization along $C$.
\end{corollary}

\begin{proof}
By Theorem \ref{theorem:stablecobound}, after stabilizations, $\widehat{\alpha}$ and $\widehat{\beta}$ cobound annuli $A$. As we have seen in Step (i) of the proof of Theorem \ref{theorem:stablecobound}, by stabilizing $\widehat{\alpha}$, we may eliminate all negative elliptic points. Dually, by stabilizing  $\widehat{\beta}$ we eliminate all positive elliptic points, hence eventually $\F(A)$ consists of s-arcs, so two boundaries of $A$ are braid isotopic.  
\end{proof}

We point out how to read the difference of self-linking number from cobounding annuli (Compare with Proposition \ref{prop:slform}).

\begin{proposition}
\label{prop:sldiff}
Assume that two closed braids $\widehat{\alpha}$ and $\widehat{\beta}$ cobound annuli $A$ admitting an open book foliation. Then 
\[
 sl(\widehat{\alpha}) -sl(\widehat{\beta}) = -(e_{+}-e_{-})+(h_{+}-h_{-}).
\]
Here $e_{\pm},h_{\pm}$ denote the number of positive/negative ellipic and hyperbolic points in open book foliation $\F(A)$.
\end{proposition}
\begin{proof}
By Corollary \ref{cor:Markov}, stabilizations of $\widehat{\alpha}$ and $\widehat{\beta}$ remove all singular points on $A$, and gives rise to a common stabilization, say $\widehat{\gamma}$.
Let $a_{\pm}$ (resp. $b_{\pm}$) be the number of positive and negative stabilizations to get $\widehat{\gamma}$ from $\widehat{\alpha}$ (resp. $\widehat{\beta}$). Since the positive stabilization preserves the self-linking number whereas the negative stabilization decreases the self-linking number by two, $sl(\widehat{\gamma}) = sl(\widehat{\alpha}) - 2 a_{-} = sl(\widehat{\beta})-2b_{-}$ hence $ sl(\widehat{\alpha}) -sl(\widehat{\beta}) = 2(a_{-}-b_{-})$.

On the other hand, one positive (resp. negative) stabilization of $\widehat{\alpha}$ removes one negative elliptic point and one negative (resp. positive) hyperbolic point. Similarly, one positive (resp. negative) stabilization of $\widehat{\beta}$ removes one positive elliptic point and one positive (resp. negative) hyperbolic point.
This implies 
\[ e_{-}=a_{+} + a_{-}, \  e_{+}=b_{+} + b_{-} ,\  h_{+}=a_{-} + b_{+}, \ h_{-}=a_{+} + b_{-}, \]
which show
\[ a_{-}-b_{-} = -e_{+}+h_{+} = e_{-}-h_{-}.\]

\end{proof}

\section{Proof of generalization of Jones-Kawamuro conjecture }
\label{sec:proof}

In this section we prove a generalization of Jones-Kawamuro conjecture. 
We prove the following theorem.

\begin{theorem}
\label{theorem:commondestab}
Let $\widehat{\alpha}$ and $\widehat{\beta}$ be closed braids in an open book $(S,\phi)$ that cobound pairwise disjoint embedded annuli $A$. Assume that the cobounding annuli $A$ and the open book $(S,\phi)$ satisfies the following three conditions.
\begin{description}
\item[C-Top] All intersections between $A$ and the binding $B$ lie on the distinguished binding component $C$. 
\item[Planar] The page $S$ is planar.
\item[FDTC] $|c(\phi,C)| >1$.
\end{description}

Then there exists closed braids $\widehat{\alpha_{0}}$ and  $\widehat{\beta_{0}}$ such that
\begin{enumerate}
\item $\widehat{\alpha_0}$ is obtained from $\widehat{\alpha}$ by braid isotopy, exchange moves and destabilizations along $C$.
\item $\widehat{\beta_0}$ is obtained from $\widehat{\beta}$ by braid isotopy, exchange moves and destabilizations along $C$.
\item $n(\widehat{\alpha_0}) = n(\widehat{\beta_0})$ and $sl(\widehat{\alpha_0})= sl(\widehat{\beta_0})$.
\end{enumerate}
\end{theorem}

The assumptions {\bf [C-Top]} and {\bf [Planar]} leads to the following properties of open book foliations of cobounding annuli, which allow us to perform b-arc foliation change freely.

\begin{lemma}
\label{lemma:key}
Under the assumptions of {\bf [C-Top]} and {\bf [Planar]}, 
\begin{enumerate}
\item All b-arcs of $A$ are separating.
\item If $v$ is an elliptic point such that all leaves that end at $v$ are b-arcs,  then around $v$ there must be both positive and negative hyperbolic points \cite[Lemma 7.6]{ik3}.
\end{enumerate}
\end{lemma}

Note that the statement (1) is nothing but a simple fact that if two endpoints of a properly embedded arc $b$ in a planar surface lie on the same component, then $b$ is separating. Also, the assertion (2) essentially follows from (1).

In the proof of Theorem \ref{theorem:commondestab}, we need to treat cobounding annuli with c-circles (see Remark \ref{rem:c-circle}), and we use the following two results, which will be proven in Section \ref{sec:c-circles}. 

First, we observe c-circles should be essential in the cobounding annuli $A$.

\begin{lemma}
\label{lemma:essentialc}
Under the assumptions of Theorem \ref{theorem:commondestab}, the cobounding annuli $A$ does not contain a c-circle which is null homotopic in $A$.
\end{lemma}

Second, we observe that for a planar open book, if cobounding annuli with c-circles is the simplest, then Theorem \ref{theorem:commondestab} is true.

\begin{lemma}
\label{lemma:degac}
Let $(S,\phi)$ be a planar open book.
Assume that closed braids $\widehat{\alpha}$ and $\widehat{\beta}$ representing a knot cobound an annulus $A$ consisting of two degenerated ac-annuli (see Figure \ref{fig:twodegac}). Then $n(\widehat{\alpha})=n(\widehat{\beta})=1$ and $sl(\widehat{\alpha})=sl(\widehat{\beta})$.

\begin{figure}[htbp]
\begin{center}
\includegraphics*[bb= 196 598 401 708,width=55mm]{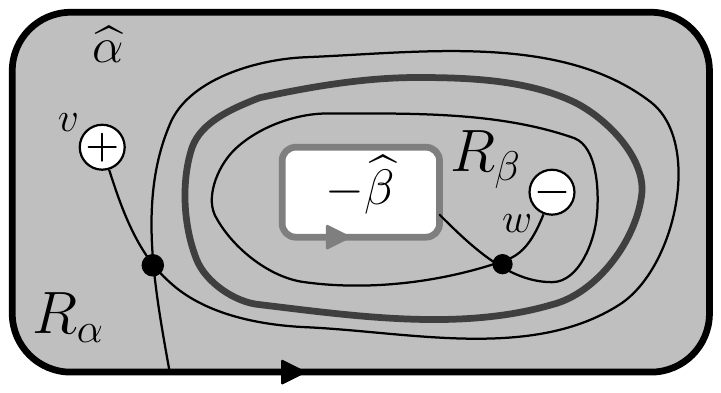}
\caption{Special case: cobounding annulus consisting of two degenerated ac-annuli}
\label{fig:twodegac}
\end{center}
\end{figure}
\end{lemma}

Using these results, we now prove Theorem \ref{theorem:commondestab}.

\begin{proof}[Proof of Theorem \ref{theorem:commondestab}]

In the following proof, we prove the theorem for the case $\widehat{\alpha}$ and $\widehat{\beta}$ are knots, namely, $A$ is an annulus, since the general link case follows from the component-wise argument.

Let us put $A$ so that it admits an open book foliation. We prove theorem by induction on the number of singular points in $A$. If $A$ contains no singular points, then $\widehat{\alpha}$ and $\widehat{\beta}$ are braid isotopic so the result is trivial.

First assume that $A$ contains c-circles. By Lemma \ref{lemma:essentialc}, c-circles are homotopic to the core of $A$. In particular, $A$ has no cc-pants or cs-annuli, and an ac- or bc- annulus always appears in pairs sharing their c-circle boundaries. 

Then, in a neighborhood of a c-circle, there is a sub-annulus $A' \subset A$ which consists of two degenerated ac-annuli (see Figure \ref{fig:thinannuli}).
Let $\partial A' = \widehat{\alpha}' \cup (-\widehat{\beta}')$. Since c-circle is homotopic to the core of $A$, the sub-annulus $A'$ splits the annulus $A$ into three cobounding annuli $A= A_{\alpha} \cup A' \cup A_{\beta}$, with $\partial A_{\alpha} = \widehat{\alpha} \cup (-\widehat{\alpha}')$, $\partial A_{\beta} = \widehat{\beta}' \cup (-\widehat{\beta})$. By Lemma \ref{lemma:degac}, $n(\widehat{\alpha}')= n(\widehat{\beta}')=1$ and $sl(\widehat{\alpha}')= sl(\widehat{\beta}')$. In particular, $\widehat{\alpha}'$ and $\widehat{\beta}'$ never admits destabilizations.

\begin{figure}[htbp]
\begin{center}
\includegraphics*[bb= 196 639 384 728,width=60mm]{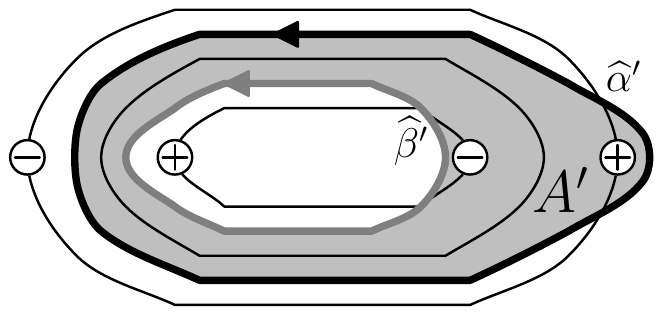}
\caption{If $\F(A)$ contains c-circles, we may find a sub-annulus $A'$ consisting of two degenerated ac-tiles.}
\label{fig:thinannuli}
\end{center}
\end{figure}

Since $A_{\alpha}$ and $A_{\beta}$ are sub-annuli of $A$, the number of singular points of $\F(A_{\alpha})$ and $\F(A_{\beta})$ is strictly smaller than that of $\F(A)$.
Therefore by induction, there exists a closed braid $\widehat{\alpha_{0}}$ (resp. $\widehat{\beta_{0}}$) which is obtained from $\widehat{\alpha}$ (resp. $\widehat{\beta}$) by braid isotopy, destabilizations and exchange moves along $C$, such that
\[ n(\widehat{\alpha_{0}}) = n(\widehat{\alpha}') = 1 =  n(\widehat{\beta}') = n(\widehat{\beta_{0}}), \ \  sl(\widehat{\alpha_{0}}) = sl(\widehat{\alpha}')= sl (\widehat{\beta}') = sl(\widehat{\beta_{0}}). \]
This completes the proof for the case $A$ contains c-circles.

Therefore we will always assume that $\F(A)$ has no c-circles. Then the region decomposition of $A$ only consists of regions which are homeomorphic to 2-cells. By collapsing the boundaries $\widehat{\alpha}$ and $\widehat{\beta}$ to a point $v_{\alpha}$ and $v_{\beta}$ respectively, we get a sphere $\mathcal{S}$, and the region decomposition of $A$ induces a cellular decomposition of $\mathcal{S}$: the 0-cell (vertices) are elliptic points and $v_{\alpha}$ and $v_{\beta}$, and the 1-cells are either a-arc, b-arc or s-arc that are the boundaries of regions, and the 2-cells are aa-, ab-, as-, abs-, or bb-tiles.

Let $V,E$ and $R$ be the number of 0-, 1- and 2-cells. We say an elliptic point $v$  is of \emph{type $(a,b)$} if, in the cellular decomposition, $v$ is the boundary of $a$ 1-cells which are a-arcs and $b$ 1-cells which are b-arcs. Let $V(a,b)$ be the number of elliptic points of $\F(A)$ which are of type $(a,b)$. Then
\begin{equation}
\label{eqn:eul0}
 V= 2 + \sum_{v=1}^{\infty}\sum_{a=0}^{v} V(a,v-a). 
\end{equation}
where the first $2$ comes from 0-cells $v_{\alpha}$ and $v_{\beta}$.

In the cellular decomposition, every 2-cell has four 1-cells as its boundary, and every 1-cell is adjacent to exactly two 2-cells, so $2R=E$ holds. Thus, by combining the equality $V-E+R = \chi(\mathcal{S}) = 2$ we get
\begin{equation} 
\label{eqn:eul1}
4 = 2V - E.
\end{equation}

Next let $E_a$, $E_b$ and $E_s$ be the number of 1-cells which are a-arcs, b-arcs, and s-arcs, respectively. Each a-arc has exactly one elliptic point  as its boundary, and each b-arc has exactly two elliptic points as its boundary, so
\begin{equation} 
\label{eqn:eul2}
E_a = \sum_{v=1}^{\infty}\sum_{a=0}^{v}a V(a,v-a), \  2E_b = \sum_{v=1}^{\infty}\sum_{a=0}^{v}(v-a)V(a,v-a).
\end{equation}

Combining (\ref{eqn:eul0}), (\ref{eqn:eul1}), (\ref{eqn:eul2}) altogether, we get
\[ 0= 2 E_{s} + \sum_{v=1}^{\infty}\sum_{a=0}^{v}(v+a-4)V(a,v-a)\]
By rewriting this equality, we get the \emph{Euler characteristic equality}
\begin{equation}
\label{eqn:euler}
  2V(1,0) + V(1,1) + 2V(0,2) + V(0,3) = 2E_s + V(2,1) + 2V(3,0) + \sum_{v=4}^{\infty} \sum_{a=0}^{v} (v+a-4)V(a,v-a).
\end{equation}

Assume that the right-hand side of (\ref{eqn:euler}) is non-zero, so one of $V(1,0)$, $V(1,1)$, $V(0,2)$ and $V(0,3)$ is non-zero.\\

{\bf Case (i): $V(1,0) \neq 0$.}\\

An elliptic point of type $(1,0)$ is contained in a degenerated aa-tile. Such an elliptic point is removed by destabilization.\\

{\bf Case (ii): $V(1,1) \neq 0 $.}\\

Let $v$ be an elliptic point of type $(1,1)$, and $\varepsilon$ and $\delta$ be the signs of the hyperbolic points around $v$. If $\varepsilon \neq \delta$, then we can remove $v$ by applying the boundary-shrinking exchange move. If $\varepsilon = \delta$, then we can apply b-arc foliation change to reduce the number of hyperbolic points around $v$. As a result, we get an elliptic point of type $(1,0)$ which can be removed by destabilization, as discussed in {\bf Case (i)}.\\

{\bf Case (iii):  $V(0,2) \neq 0 $.}\\

Let $v$ be an elliptic point of type $(0,2)$, and $\varepsilon$ and $\delta$ be the signs of the hyperbolic points around $v$. By Lemma \ref{lemma:key}(2), $\varepsilon \neq \delta$. 
Moreover, $v$ cannot be strongly essential, because otherwise by Lemma \ref{lemma:FDTC} we get $-1 \leq c(\phi,C) \leq 1$, which contradicts {\bf [FDTC]}. Hence we can remove such an elliptic point by an interior exchange move.\\

{\bf Case (iv): $V(0,3) \neq 0$. }\\

Let $v$ be an elliptic point of type $(0,3)$. By Lemma \ref{lemma:key}, around $v$ there must be both positive and negative hyperbolic points. Around $v$ there are three hyperbolic points, so we can find two hyperbolic points of the same sign which are adjacent, so the b-arc foliation change can be applied. After the b-arc foliation change, we get an elliptic point of type $(0,2)$ which can be removed as discussed in {\bf Case (iii)}.\\

By {\bf Case (i)--(iv)} above, if the right-hand side of (\ref{eqn:euler}) is non-zero, then we can reduce the number of singular points of $\F(A)$ hence by induction, we find the desired closed braids.

Therefore we now assume that the right hand side of (\ref{eqn:euler}) is zero.
Thus, $V(1,0)=V(1,1)= V(0,2) =V(0,3) = E_s=0$ and all elliptic points are either of type $V(1,2)$ or $V(0,4)$.

Assume that around some elliptic point $v$ of type $(0,4)$, the sign of hyperbolic points are not alternate. This means that there is a situation where the b-arc foliation change can be applied, and applying the b-arc foliation change, $v$ is changed to be of type $(0,3)$, which can be removed by {\bf Case (iv)}.

Next we look at an elliptic point $v$ of type $(1,2)$. Such $v$ lies in one bb-tile $R_{bb}$ and two ab-tiles $R_{ab}^{1}$, $R_{ab}^{2}$. 
Assume that $\sgn(R_{ab}^{1}) \neq \sgn(R_{ab}^{2})$, or $\sgn(R_{ab}^{1})= \sgn(R_{ab}^{2}) = \sgn(R_{bb})$. Then by b-arc foliation change we get an elliptic point of type $(1,1)$, which can be removed by {\bf Case (ii)}.

Therefore unless the open book foliation $\F(A)$ is in a particular form,  a  tiling with alternate signs (see Figure \ref{fig:standannuli}(i) -- around each elliptic point of type $(0,4)$ the sign of hyperbolic points are alternate and around each elliptic point of type $(1,2)$, $\sgn(R_{ab}^{1})= \sgn(R_{ab}^{2}) \neq \sgn(R_{bb})$), we can reduce the number of singular point of $\F(A)$. 

Thus, we eventually reduced the proof for the case that the cobounding annulus $A$ is tiled with alternate signs.
Let $\varepsilon$ be the sign of ab-tile containing $\widehat{\alpha}$. If $\varepsilon = +$ (resp. $-$), then by negatively (resp. positively) stabilizing $\widehat{\alpha_0}$ we eliminate all negative elliptic points and by negatively (positively) stabilizing $\widehat{\beta}$, we eliminate all positive elliptic points (see Figure \ref{fig:standannuli} (ii)) to get isotopic closed  braids.
The observation that $A$ is tiled with alternate signs implies that the number of necessary stabilizations are the same, so $sl(\widehat{\alpha})= sl(\widehat{\beta})$ and $n(\widehat{\beta}) = n(\widehat{\alpha})$.

\begin{figure}[htbp]
\begin{center}
\includegraphics*[bb= 85 566 514 740,width=140mm]{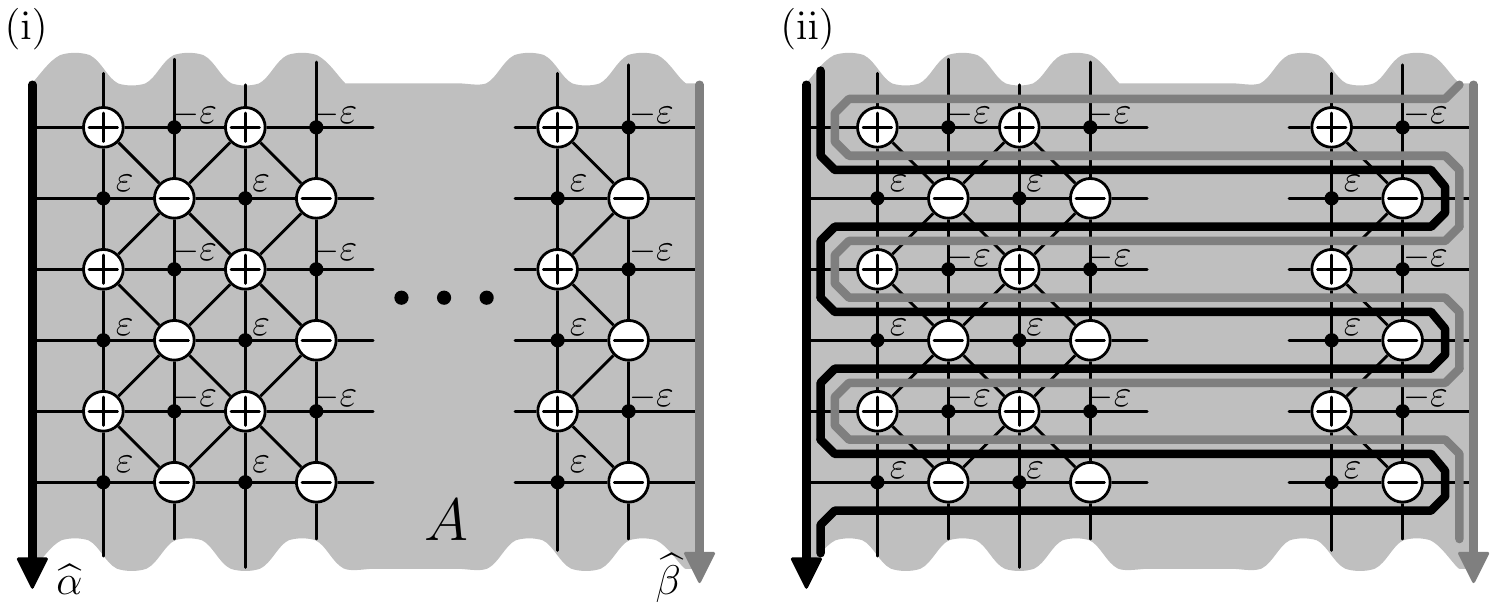}
\caption{(i) Cobounding annuli $A$ with special open book foliation, a tiling with alternate signs. (ii) The boundaries $\widehat{\alpha}$ and $\widehat{\beta}$ become braid isotopic by performing the stabilization of sign $-\varepsilon$ same times.}
\label{fig:standannuli}
\end{center}
\end{figure}
\end{proof}

\begin{remark}
\label{rem:c-circle}
By Proposition \ref{prop:withoutc}, we may always assume that the first cobounding annulus $A$ has no c-circles. However, when we apply the interior exchange move ({\bf Case (iii)}) one may encounter cobounding annulus with c-circles. Such c-circles cannot be eliminated without increasing the number of singular points. This is a reason why we need to treat open book foliation with c-circles. In the braid foliation case, this problem does not occur since one can always remove c-circles without increasing the number of singular points.
\end{remark}

The Jones-Kawamuro conjecture is a direct consequence of Theorem \ref{theorem:stablecobound} and Theorem \ref{theorem:commondestab}.

\begin{proof}[Proof of Theorem \ref{theorem:main}]

Assume to the contrary that there exist closed braids $\widehat{\alpha}$ and $\widehat{\beta}$ which are $C$-topologically isotopic to the same link $L$ violating the inequality (\ref{eqn:themain}):
\[
|sl(\widehat{\alpha})-sl(\widehat{\beta})| > 2( \max\{ n(\widehat{\alpha}), n(\widehat{\beta})\} - b_{C}(L)). \]
 
With no loss of generality, we assume that $sl(\widehat{\alpha}) \geq sl(\widehat{\beta})$.
By Theorem \ref{theorem:stablecobound}, there exists closed braids $\widehat{\alpha_{+}}$ and $\widehat{\beta_{-}}$ that cobound annuli $A$, where $\widehat{\alpha_{+}}$ is a positive stabilizations of $\widehat{\alpha}$, and $\widehat{\beta_{-}}$ is a negative stabilizations of $\widehat{\beta}$ along the distinguished binding component $C$. By taking further negative stabilizations of $\widehat{\beta}$ if necessary, we may assume that $n(\widehat{\beta_{-}}) \geq n(\widehat{\alpha_{+}})$.

Since a positive stabilization preserves the self-linking number whereas one negative stabilization decreases the self-linking number by two, we have
\[ sl(\widehat{\alpha_{+}})-sl(\widehat{\beta_{-}}) = sl(\widehat{\alpha}) - sl(\widehat{\beta}) + 2(n(\widehat{\beta_{-}}) - n(\widehat{\beta})). \]
This shows that $\widehat{\alpha_{+}}$ and $\widehat{\beta_{-}}$ also violate the inequality (\ref{eqn:themain}), namely,
\begin{equation}
\label{eqn:contradict}
|sl(\widehat{\alpha_+})-sl(\widehat{\beta_-})| = sl(\widehat{\alpha_+})-sl(\widehat{\beta_-}) > 2( \max\{n(\widehat{\alpha_+}),n(\widehat{\beta_{-}})\}) - b_{C}(K)=2n(\widehat{\beta_-}) -2b_{C}(L).
\end{equation}

Since $\widehat{\alpha}$ and $\widehat{\beta}$ are $C$-topologically isotopic, the cobounding annuli $A$ between $\widehat{\alpha_{+}}$ and $\widehat{\beta_{-}}$ can be chosen so that the assumption {\bf [C-Top]} in Theorem \ref{theorem:commondestab} is satisfied. Hence by Theorem \ref{theorem:commondestab}, there are closed braids $\widehat{\alpha_{0}}$ and $\widehat{\beta_{0}}$ with $n(\widehat{\alpha_0})=n(\widehat{\beta_0})$ and $sl(\widehat{\alpha_0})=sl(\widehat{\beta_0})$, obtained from $\widehat{\alpha_+}$ and $\widehat{\beta_-}$ by destabilizations and exchange moves along $C$.
Since exchange move preserves the self-linking number we have
\[ 
-2( n(\widehat{\alpha_+}) - n(\widehat{\alpha_0}) ) \leq sl(\widehat{\alpha_+}) -sl(\widehat{\alpha_0}) \leq 0, \ -2( n(\widehat{\beta_+}) - n(\widehat{\beta_0}) ) \leq sl(\widehat{\beta_+}) - sl(\widehat{\beta_0}) \leq 0
\]
Hence
\[ -2( n(\widehat{\alpha_+}) - n(\widehat{\alpha_0}) ) \leq sl(\widehat{\alpha_+}) -sl(\widehat{\beta_-}) \leq 2( n(\widehat{\beta_+}) - n(\widehat{\beta_0})).
\]
This contradicts with (\ref{eqn:contradict}).
\end{proof}

\section{Cobounding annuli with c-circles}
\label{sec:c-circles}
In this section we prove results on cobounding annuli with c-circles used in the previous section.

\begin{proof}[Proof of Lemma \ref{lemma:essentialc}]

The proof is essentially the same as an argument already appeared in \cite[pp. 3016 Case II]{ik3}, the proof of split closed braid theorem for the case that a splitting sphere contains c-circles.

Assume that the cobounding annuli $A$ contain a c-circle which is null-homotopic. 
Take an innermost bc-annulus $R$. Here by `innermost' we mean that the c-circle boundary of $R$ bounds a disc $D \subset A$ with $R \subset D$ so that $D-R$ contains no c-circles. Then either $R$ is degenerate bc-annulus (see Figure \ref{fig:degenerated}), or, the region decomposition of $D-R$ consists only of bb-tiles. 
We prove the lemma by induction of the number of bb-tiles in $D-R$.

First assume that $D-R$ contains no bb-tiles, namely, $R$ is degenerated. Take a binding component $C$ so that one of the elliptic points in $R$ lies on $C$. Then by \cite[Lemma 7.7]{ik3}, $|c(\phi,C)| \leq 1$ which is a contradiction.

Assume that $D-R$ contains $k>0$ elliptic points, and let $v_{\pm}$ be the elliptic points which lie on $R$. Let us consider the 2-sphere $\mathcal{S}$ obtained by gluing two b-arc boundaries of $D-R$. Then the region decomposition of $D-R$ induces a cellular decomposition of $\mathcal{S}$. 

For $i>0$, let $V(i)$ be the number of 0-cells of valence $i$ in the cellular decomposition of $\mathcal{S}$. Then by a similar argument as the equation (\ref{eqn:euler}) in the proof of Theorem \ref{theorem:commondestab}, we have the Euler characteristic equality
\[ 2V(2) + V(3) = 8 + \sum_{i\geq 4}(i-4)V(i). \]
This shows that $\mathcal{S}$ has a 0-cell $v$ (elliptic point) of valence $\leq 3$ which is not $v_{\pm}$.

By applying b-arc foliation change if necessary (thanks to Lemma \ref{lemma:key}, this is always possible) we may assume that $v$ is of valence two (cf. {\bf Case (iv)} in the proof of Theorem \ref{theorem:commondestab}). Then as we have discussed, by interior exchange move we can remove elliptic point $v$. Hence we can reduce the number of elliptic points in $D-R$, so by induction we conclude that a null-homotopic c-circle never exists.
\end{proof}

Next we prove Lemma \ref{lemma:degac}. As we will see in Lemma \ref{lemma:cexam0} and Example \ref{example:c}, Lemma \ref{lemma:degac} does not hold for non-planar open books.

\begin{proof}[Proof of Lemma \ref{lemma:degac}]
Let $A=R_{\alpha} \cup R_{\beta}$, where $R_{\alpha}$ and $R_{\beta}$ are degenerated ac-annuli containing $\widehat{\alpha}$ and $\widehat{\beta}$, respectively, and let $v$ and $w$ be the positive and negative elliptic points in $\F(A)$. (See Figure \ref{fig:twodegac} again). 

Since there are no b- and s-arcs in $\F(A)$, $n(\widehat{\alpha})$ and $n(\widehat{\beta})$ are equal to the number of positive and negative elliptic points so $n(\widehat{\alpha})=n(\widehat{\beta})=1$. We look at the movie presentation of $A$ to determine the closed braids $\widehat{\alpha}$ and $\widehat{\beta}$.

We denote the a-arcs in a page $S_t$ whose endpoints are $v$ and $w$ by $a_v=a_v(t)$ and $a_w =a_{w}(t)$, respectively. 
Take $S_{0}$ so that the number of c-circles in $S_{0}$ is minimal among all $S_{t}$ $(t\in [0,1])$. Then $S_{0} \cap A$ consists of two a-arcs $a_{w}(0), a_v(0)$ and c-circles $c_1,\ldots,c_k$. We denote the c-circle in $S_t$ that corresponds to $c_i$ by $c_i(t)$, or simply by $c_i$.

Take $t_{\alpha},t_{\beta} \in [0,1]$ so that $S_{t_{\alpha}}$ and $S_{t_{\beta}}$ are the singular pages that contain the hyperbolic point $h_{\alpha}$ in $R_{\alpha}$ and $h_{\beta}$ in $R_{\beta}$, respectively. We treat the case $0<t_{\alpha} < t_{\beta} <1$. The case $0< t_{\beta} < t_{\alpha} < 1$ is similar. With no loss of generality, we may assume that $0 < t_{\alpha} < \frac{1}{2} < t_{\beta} <1$.

Let us look at what will happen as $t$ moves from $0$ to $1$.
Since we have assumed that the number of c-circles in $S_0$ is minimum, the first ac-singular point $h_{\alpha}$ splits the a-arc $a_{v}$ into an a-arc and a new c-circle, say $c_{k+1}$. Similarly, the second ac-singular point in $h_{\beta}$ merges the a-arc $a_w$ and one of c-circles, say $c_{i}$.
Finally, $S_1\cap A$ is identified with $S_0 \cap A$ by the monodromy $\phi:S_1 \rightarrow S_0$.

Recall that every simple closed curve in a planar surface is separating. Take $j \in \{1,\ldots,k\}$ so that $j \neq i$. Since the monodromy $\phi$ preserves $\partial S$, $c_{j}(1)$ is separating implies that $\phi(c_j(1)) = c_{j}(0)$, unless $c_j(1)$ is null-homotopic in $S_1$. However, $\phi(c_j(1)) = c_{j}(0)$ means that a family of curves $c_{j}(t)$ $(t \in [0,1])$ yields an embedded torus, which is absurd. Thus, we conclude we have either

\begin{enumerate}
\item $k=0$, that is, $S_0 \cap A$ consists of two a-arcs $a_{v}$ and $a_{w}$. 
\item All the c-circles $c \in S_t$ are null-homotopic in $S_t$.
\end{enumerate}

In the case (1), $A\cap S_{\frac{1}{2}}$ consists of two a-arcs $a_v, a_w$ and the unique c-circle $\mathcal{C}$. The movie presentation of $A$ is described as follows (see Figure \ref{fig:movieacac}).

\begin{enumerate}
\item[(i)] At $t=0$, we have two a-arcs $a_v$ and $a_w$. Here we write the future  position of c-circle $\mathcal{C}$ by gray, dotted line.
\item[(ii)] As $t$ approaches to $t_{\alpha}$, the arc $a_v(t)$ deforms to enclose the position of c-circle $\mathcal{C}$, and at $t=t_\alpha$, $a_{v}$ forms a hyperbolic point $h_{\alpha}$. At $t=t_{\alpha} + \varepsilon$ for small $\varepsilon>0$, we have an a-arc $a_v$ and a new c-circle $\mathcal{C}$.
\item[(iii),(iv)] As $t$ approaches to $\frac{1}{2}$, the point $\widehat{\alpha}\cap S_t$ moves along $a_{v}(t)$ to go back to the position at $t=0$. As a consequence, the 1-braid $\alpha$ turns around $\mathcal{C}$ once.
\item[(v)] At $t=\frac{1}{2}$, we have two a-arcs $a_v$ and $a_w$, and a c-circle $\mathcal{C}$, which is a separating simple closed curve in $S_{\frac{1}{2}}$.
\item[(vi)] As $t$ approaches to $t_{\beta}$, the arc $a_w(t)$ deforms to approaches the c-circle $\mathcal{C}$, and at $t=t_\beta$, $a_{w}$ and $\mathcal{C}$ forms a hyperbolic point $h_{\beta}$. At $t=t_{\alpha} + \varepsilon$ for small $\varepsilon>0$, c-circle $\mathcal{C}$ disappears.
\item[(vii,viii)]  As $t$ approaches to $1$, the point $\widehat{\beta}\cap S_t$ moves along $a_{w}(t)$ to go back to the position at $t=0$. As a consequence, the 1-braid $\beta$ turns around $\mathcal{C}$ once. Finally, two pages $S_1$ and $S_0$ are identified by the monodromy $\phi$.
\end{enumerate}

Thus in particular, $\mathcal{C}$ is separating implies
\begin{equation}
\label{eqn:critical}
\sgn(R_{\alpha}) \neq \sgn(R_{\beta}).
\end{equation}
By Proposition \ref{prop:sldiff} we conclude $sl(\widehat{\alpha}) = sl(\widehat{\beta})$.

In the case (2), a similar argument shows that both $\widehat{\alpha}$ and $\widehat{\beta}$ are closure of the trivial 1-braid so $sl(\widehat{\alpha}) = sl(\widehat{\beta})=-1$.

\begin{figure}[htbp]
\begin{center}
\includegraphics*[bb=126 459 493 710,width=130mm]{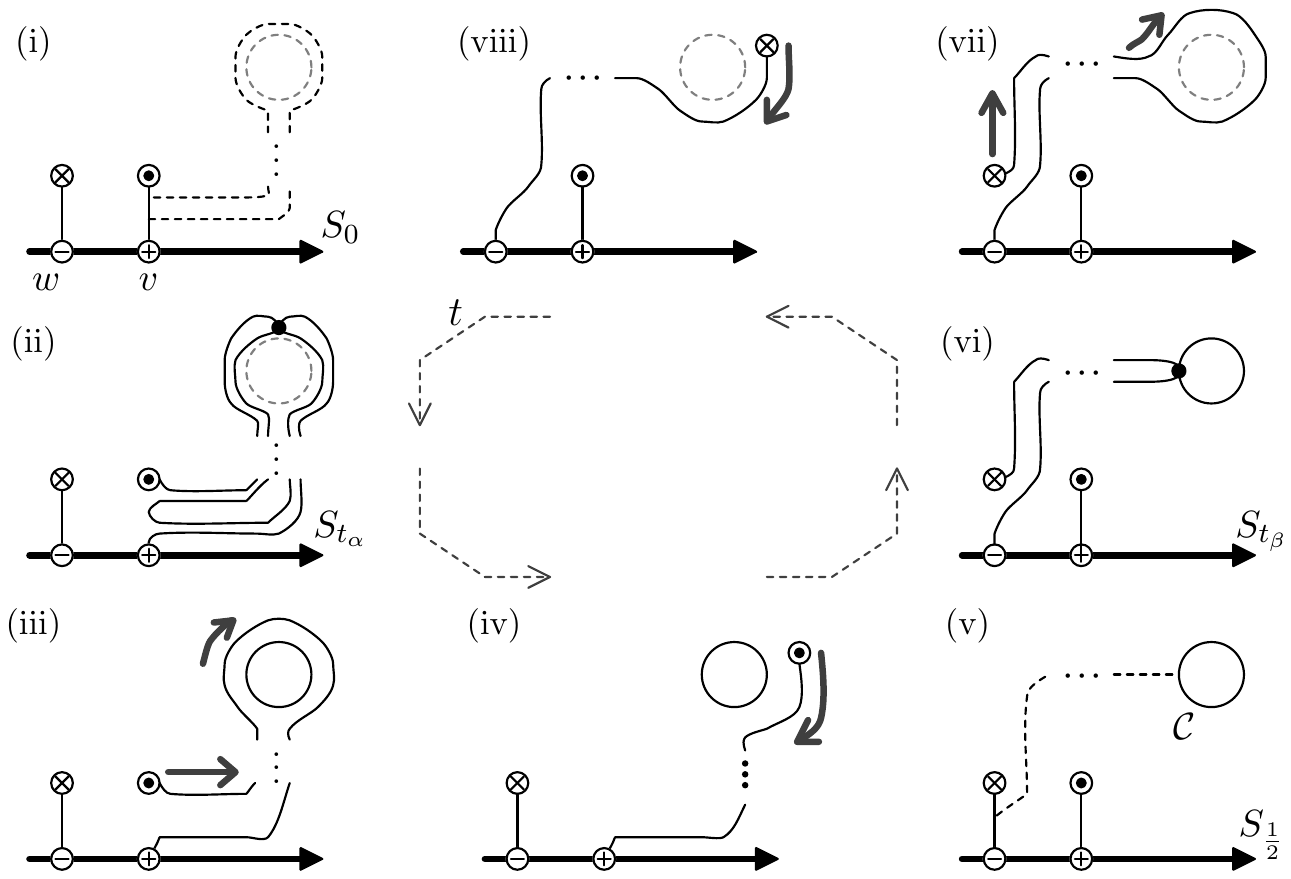}
\caption{Movie presentation of $\F(A)$}
\label{fig:movieacac}
\end{center}
\end{figure}

\end{proof}

As we will see in Example \ref{example:c} in the next section, (\ref{eqn:critical}) does not necessarily hold if a page $S$ is not planar.

\section{Counter examples of Jones-Kawamuro conjecture}
\label{sec:counterexamples}

We close the paper by giving several counter examples that illustrate the necessity of assumptions in Theorem \ref{theorem:main}.

First of all, the following example, coming from our first counterexample Example \ref{exam:counter}, shows that the FDTC assumption {\bf [FDTC]} is necessary and the inequality $>1$ is best-possible: One can not replace the condition $>1$ with $\geq 1$.

\begin{example}[Example \ref{exam:counter}, revisited]
\label{example:a}
Let $A$ be an annulus with boundary $C_1$ and $C_2$, and $T_{A}$ be the right-handed Dehn twist along the core of $A$.

Let us recall the counter example in Example \ref{exam:counter}: A closed 1-braid $\widehat{\alpha}$, the boundary of transverse overtwisted disc, and a closed 1-braid $\widehat{\beta}$, the meridian of $C_1$ in an annulus open book $(A,T_{A}^{-1})$. The binding $\partial A = C_{1} \cup C_2$ forms a negative Hopf link in $S^{3}$. Also note that $c(\phi,C_1) = c(\phi,C_2)=-1$.
 As a link in $S^{3}$, $\widehat{\alpha}$ and $\widehat{\beta}$ are depicted in Figure \ref{fig:counterexam1} (i).

Let $\widehat{\alpha}'$ (resp. $\widehat{\beta}'$) be the positive (resp. negative) stabilization of $\widehat{\alpha}$ (resp. $\widehat{\beta}$) along $C_2$. Then $\widehat{\alpha}'$ and $\widehat{\beta}'$ are $C_1$-topologically isotopic (see Figure \ref{fig:counterexam1} (ii)). On the other hand,
\[ |sl(\widehat{\alpha}') - sl(\widehat{\beta}')| = |1-(-3)| = 4 > 2 = 2( \max \{n(\widehat{\alpha}'), n(\widehat{\beta}') \} -b_{C}(K) ) \]
hence they violate the inequality (\ref{eqn:themain}). 

\begin{figure}[htbp]
\begin{center}
\includegraphics*[bb=168 572 425 736,width=80mm]{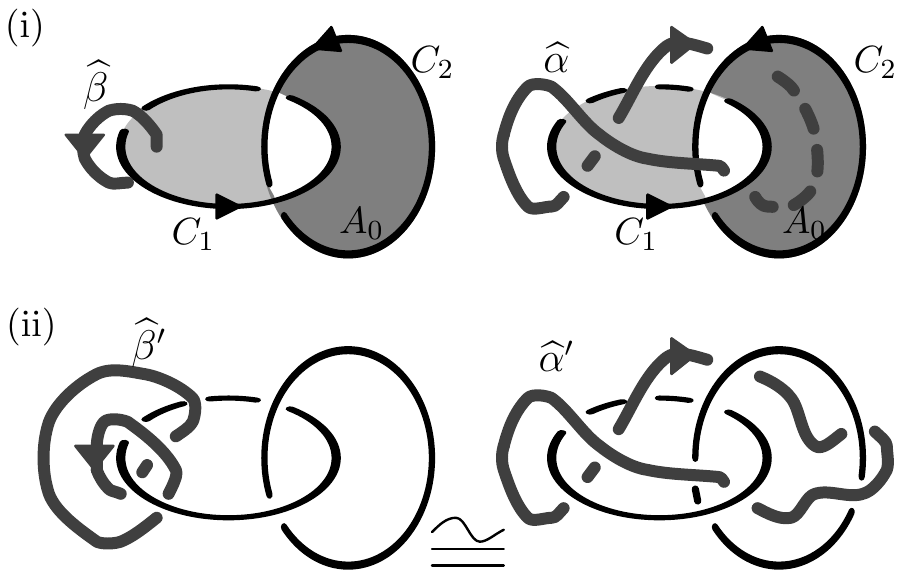}
\caption{Counter example for Jones-Kawamuro conjecture for planar open book with FDTC=1. (i) Closed braids $\widehat{\alpha}$ and $\widehat{\beta}$ in $S^{3}$, (ii) A negative stabilization of $\widehat{\alpha}$ along $C_2$ is $C_1$-topologically isotopic to $\widehat{\beta}$.}
\label{fig:counterexam1}
\end{center}
\end{figure}

\end{example}

The next example shows that the notion of $C$-topologically isotopic is also necessary.

\begin{example}
\label{example:b}
 Let us consider the open book $(A,T_{A}^{2})$, which is an open book decomposition of the unique tight (indeed, Stein fillable) contact structure of $\R P^{3}=L(2,1)$. The FDTCs are $c(T_{A}^{2},C_1) = c(T_{A}^{2},C_2)=2$, so the open book $(A,T_{A}^{2})$ satisfies two assumptions {\bf [Planar]} and {\bf [FDTC]} in Theorem \ref{theorem:main} for both $C_1$ and $C_2$.

Let $\widehat{\alpha} = \partial D$ be a closed braid which is a boundary of a disc $D$, given by the movie presentation in Figure \ref{fig:countermovie}.
From the movie presentation we read that $sl(\widehat{\alpha}) = -5$ and $n(\widehat{\alpha})=2$. On the other hand, let $\widehat{\beta}$ be a closed braid which is a meridian of $C_1$, so $sl(\widehat{\beta})= -1$ and $n(\widehat{\beta})=1$.

\begin{figure}[htbp]
\begin{center}
\includegraphics*[bb=128 540 467 734,width=120mm]{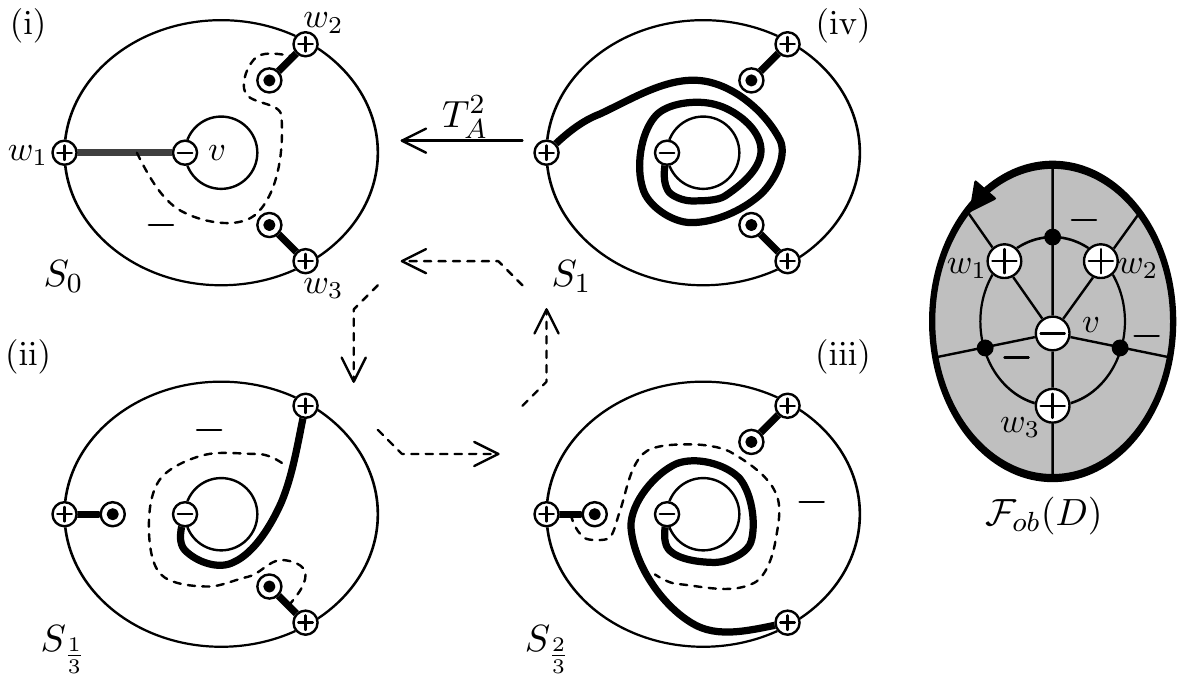}
\caption{A movie presentation of a disc $D$: The last slice (iv) at $t=1$ is identified with the first slice (i) at $t=0$ by the monodromy $T_{A}^{2}$.}
\label{fig:countermovie}
\end{center}
\end{figure}

Both $\widehat{\alpha}$ and $\widehat{\beta}$ are unknots hence they are topologically isotopic. However, 
\[ 4 = |sl(\widehat{\alpha})-sl(\widehat{\beta})| > 2 ( \max\{n(\widehat{\alpha}) ,n(\widehat{\beta})\} -1)  = 2 \]
so they violate the inequality (\ref{eqn:themain}). 

Note that if $\widehat{\alpha}$ and $\widehat{\beta}$ are $C_{1}$-topologically isotopic, then the links $\widehat{\alpha} \cup C_{2}$ and $\widehat{\beta} \cup C_2$ must be isotopic in $M_{(A,T_{A}^{2})}= \R P^{3}$, hence their linking number must be the same.
However, 
\[ 3= lk(\widehat{\alpha},C_1) \neq lk(\widehat{\alpha},C_1) = 1, \ \ -1= lk(\widehat{\alpha},C_2) \neq lk(\widehat{\beta},C_2) = 0.\]
hence $\widehat{\alpha}$ and $\widehat{\beta}$ are neither $C_1$-topologically isotopic nor $C_{2}$-topologically isotopic,

\end{example}

Actually as the next proposition shows, similar counter examples are quite ubiquitous. This shows that in Theorem \ref{theorem:main} the minimal $C$-braid index $b_{C}(K)$ cannot be replaced with the usual minimal braid index $b(K)$, the minimum number of strands needed to represent $K$ as a closed braid in $M_{(S,\phi)}$.

\begin{proposition}
\label{prop:counter}
Let $S$ be a (not necessarily planar) surface with more than one boundary components. For arbitrary open book $(S,\phi)$ with $\phi \neq \textsf{Id}$, there are two closed braids $\widehat{\alpha}$ and $\widehat{\beta}$ in $M_{(S,\phi)}$ which represents the unknot (hence they are topologically isotopic) but they violate the inequality (\ref{eqn:themain}). 
\end{proposition}
\begin{proof}
Take two different boundary components $C_1$ and $C_2$ of $S$.
By applying a construction in \cite[Theorem 2.4]{ik1-2}, if $\phi \neq \textsf{Id}$, one gets an embedded disc $D$ admitting open book foliation with the following properties (see Figure \ref{fig:cexamD}).
\begin{enumerate}
\item $\F(D)$ has unique negative elliptic point $v$ which lies on $C_1$ and $n(>1)$ positive elliptic points $w_1,\ldots,w_n$ $(n\geq 2)$ which lie on $C_2$.
\item The region decomposition of $\F(D)$ consists of $n$ ab-tiles of the same sign $\varepsilon$. $\varepsilon = +$ (resp. $\varepsilon = -$) if $\phi$ is not right-veering (resp. right-veering) at $C_1$.  
\end{enumerate}

\begin{figure}[htbp]
\begin{center}
\includegraphics*[bb=216 628 380 727,width=50mm]{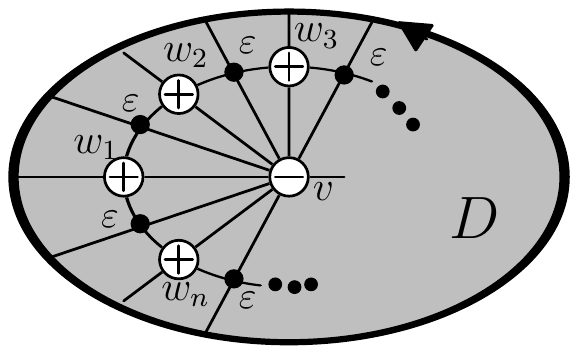}
\caption{The disc $D$ and its open book foliation. In the case $\varepsilon = +$, $D$ is a transverse overtwisted disc.}
\label{fig:cexamD}
\end{center}
\end{figure}
Let $\widehat{\alpha}= \partial D$. Then $n(\widehat{\alpha})=n-1$ and $sl(\widehat{\alpha}) = -(n-1) + \varepsilon n$.

In the case $\varepsilon = +$, let $\widehat{\beta}$ be a closed $(n-1)$-braid which is obtained from the meridian of $C_1$ by negatively stabilizing $(n-2)$ times along $C_2$.
Then $sl(\widehat{\beta})= 3-2n$ and $n(\widehat{\beta})= n-1$, hence they violate the inequality (\ref{eqn:themain}),
\[ 2n-2= |sl(\widehat{\alpha})-sl(\widehat{\beta})| > 2 ( \max\{n(\widehat{\alpha}) ,n(\widehat{\beta})\} -1)  = 2n-4. \]

In the case $\varepsilon = -$, let $\widehat{\beta}$ be a closed $1$-braid which is a meridian of $C_1$. Then $sl(\widehat{\beta})= -1$ and $n(\widehat{\beta})= 1$, hence they violate the inequality (\ref{eqn:themain}),
\[ 2n-2= |sl(\widehat{\alpha})-sl(\widehat{\beta})| > 2 ( \max\{n(\widehat{\alpha}) ,n(\widehat{\beta})\} -1)  = 2n-4. \]

As in Example \ref{example:b}, one can check that $\widehat{\alpha}$ and $\widehat{\beta}$ are not $C$-isotopic for any every boundary component $C$ of $S$, by looking at the linking number with $C_1$ and $C_2$.

\end{proof}

To illustrate the necessity of planarity, we give a counter example of the property (\ref{eqn:critical}) appeared in the proof Lemma \ref{lemma:degac}.

\begin{lemma}
\label{lemma:cexam0}
Let $S$ be non-planar surface. Then for arbitrary monodromy $\phi$, there exist closed 1-braids $\widehat{\alpha}$ and $\widehat{\beta}$ and a cobounding annulus $A$ between them in $M_{(S,\phi)}$, such that
\begin{enumerate}
\item The region decomposition of $A$ consists of two degenerated ac-annuli $R_{\alpha}$ and $R_{\beta}$ (see Figure \ref{fig:twodegac} again).
\item $\sgn(R_{\alpha}) = \sgn(R_{\beta})$.
\end{enumerate}
\end{lemma}
\begin{proof}
We give such a cobounding annulus $A$ by movie presentation. Here we give an example $\sgn(R_{\alpha}) = \sgn(R_{\beta})=+$. An example of $\sgn(R_{\alpha}) = \sgn(R_{\beta})=-$ is obtained similarly. See Figure \ref{fig:movieacac2}. Note that the movie is quite similar to Figure \ref{fig:movieacac}, and the main difference is the slice (v), where the description arc of the hyperbolic point shows $R_{\beta} = +$.

\begin{figure}[htbp]
\begin{center}
\includegraphics*[bb= 116 481 486 730,width=130mm]{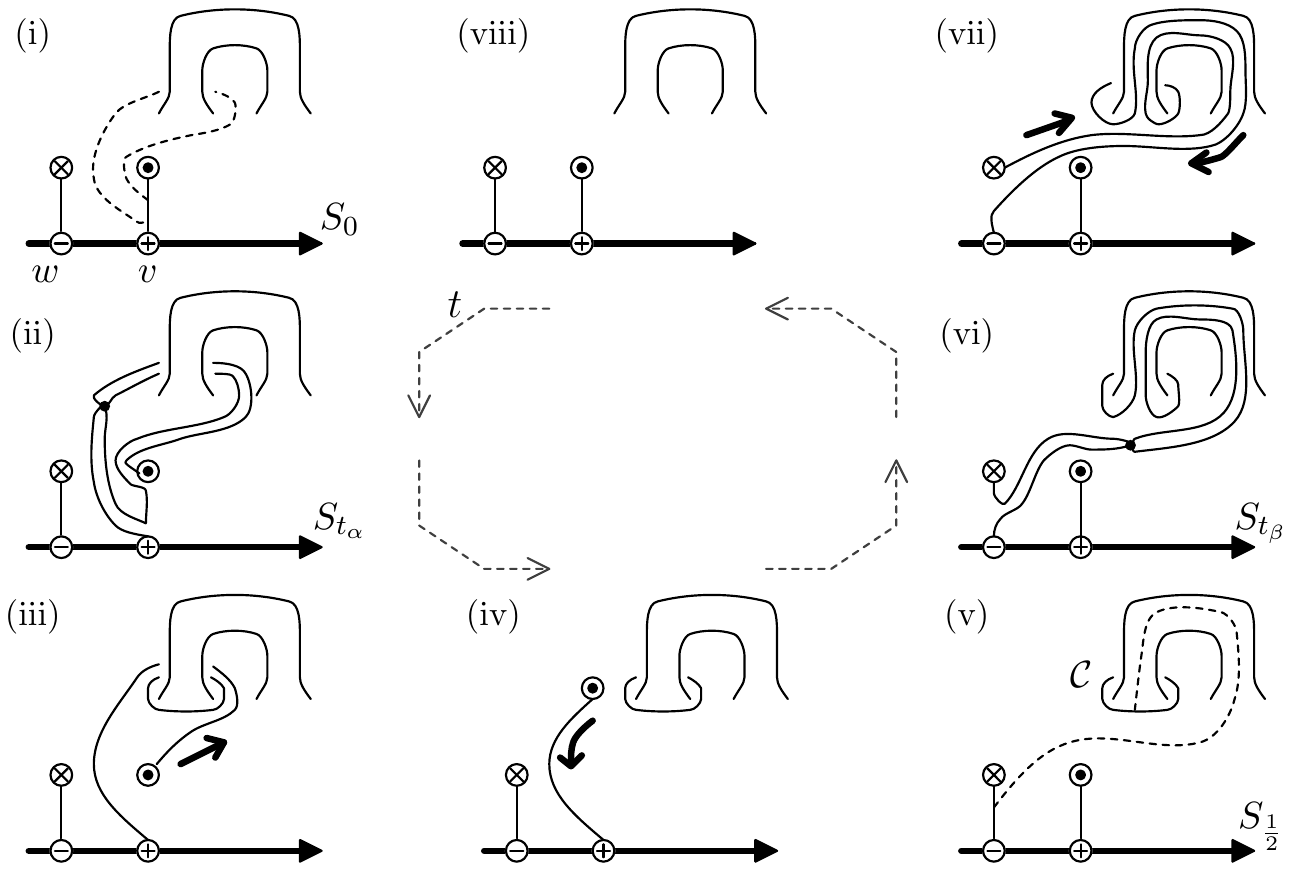}
\caption{Movie presentation of a cobounding annulus $A$ consisting of two degenerated ac-annuli with the \emph{same} sign (compare with Figure \ref{fig:movieacac}) }
\label{fig:movieacac2}
\end{center}
\end{figure}

\end{proof}

A cobounding annulus $A$ in Lemma \ref{lemma:cexam0} gives a counter example of the  Jones-Kawamuro conjecture for non-planar open books.

\begin{example}
\label{example:c}
Let $S$ be the once-hold surface of genus $>0$ and take a monodromy $\phi$ so that $M_{(S,\phi)}$ is an integral homology sphere, with $|c(\phi,\partial S)| >1$.

Let $\widehat{\alpha}$ and $\widehat{\beta}$ be closed 1-braidsgiven by the movie presentation Figure \ref{fig:movieacac2} in Lemma \ref{lemma:cexam0}.
Then $\widehat{\alpha}$ and $\widehat{\beta}$ are $\partial S$-topologically isotopic and null-homologous in $M_{(S,\phi)}$, but by Proposition \ref{prop:sldiff}
 \[ sl(\widehat{\alpha})-sl(\widehat{\beta}) = 2 > 2(\max \{n(\widehat{\alpha},n(\widehat{\beta})\} -1) = 0. \]

\end{example}

\section*{Acknowledgements}
The author gratefully thanks Keiko Kawamuro. Without a series of collaborations with her, he cannot write the present paper, and this work is philosophically a continuation of a joint work with her. He also thanks Inanc Baykur, John Etnyre, Jeremy Van Horn-Morris for sharing Example \ref{exam:counterBEHK}.
The author was partially supported by JSPS KAKENHI
Grant Numbers 25887030 and 15K17540.

\end{document}